\documentclass[11pt,leqno]{amsart}
\usepackage{amsmath}
\usepackage{amsfonts}
\usepackage{amssymb}
\usepackage{graphicx}
\usepackage{hyperref}
\usepackage{xcolor} 
\usepackage{a4wide} 
\usepackage{mathrsfs} 
\usepackage{relsize} 

\usepackage{soul}

\theoremstyle{plain}
\newtheorem{theorem}{Theorem}

\newtheorem{corollary}[theorem]{Corollary}
\newtheorem{definition}[theorem]{Definition}
\newtheorem{lemma}[theorem]{Lemma}
\newtheorem{proposition}[theorem]{Proposition}

\theoremstyle{definition}
\newtheorem{example}[theorem]{Example}

\newtheorem{remark}[theorem]{Remark}

\newtheorem{theoremA}{Theorem}

\numberwithin{equation}{section}
\numberwithin{theorem}{section}

\newcommand{\metric}{\<\cdot,\cdot\>}

\newcommand{\vol}{\mathrm{vol}}

\renewcommand{\div}{\operatorname{div}}

\newcommand{\Hess}{\operatorname{Hess}}

\newcommand{\Lip}{\mathrm{Lip}}

\newcommand{\rr}{\mathbb{R}}
\renewcommand{\ss}{\mathbb{S}}

\newcommand{\sect}{\operatorname{Sect}}
\newcommand{\ric}{\operatorname{Ric}}

\newcommand{\supp}{\operatorname{supp}}

\newcommand{\<}{\langle}
\renewcommand{\>}{\rangle}
\newcommand{\dv}{\,d\mathrm{vol}}
\newcommand{\CAP}{\mathrm{Cap}}

\def\XXint#1#2#3{{\setbox0=\hbox{$#1{#2#3}{\int}$}
     \vcenter{\hbox{$#2#3$}}\kern-.5\wd0}}

\newcommand{\cut}{\mathrm{cut}}

\newcommand{\inj}{\mathrm{inj}}

\renewcommand{\a}{\alpha}
\renewcommand{\b}{\beta}

\renewcommand{\d}{\delta}

\newcommand{\g}{\gamma}
\renewcommand{\l}{\lambda}

\newcommand{\s}{\sigma}
\newcommand{\vp}{\varphi}

\newcommand{\CC}{\mathcal{C}}

\newcommand{\CG}{\mathcal{G}}

\newcommand{\A}{\mathscr{A}}

\newcommand{\D}{\mathscr{D}}
\newcommand{\E}{\mathscr{E}}
\renewcommand{\H}{\mathscr{H}}
\renewcommand{\L}{\mathscr{L}}

\newcommand{\Hm}[1]{\leavevmode{\marginpar{\tiny%
			$\hbox to 0mm{\hspace*{-0.5mm}$\leftarrow$\hss}%
			\vcenter{\vrule depth 0.1mm height 0.1mm width \the\marginparwidth}%
			\hbox to 0mm{\hss$\rightarrow$\hspace*{-0.5mm}}$\\\relax\raggedright
			#1}}}
\begin{document}

\title[Volume and parabolicity for non-symmetric operators]{Volume growth and parabolicity for non-symmetric operators}

\author{Leandro F. Pessoa}
\address{Universidade Federal do Piau\'i \\ Departamento de Matematica \\ 64049-550, Teresina, BRAZIL}
\email{leandropessoa@ufpi.edu.br}
\author{Stefano Pigola}
\address{Universit\`a degli Studi di Milano-Bicocca\\ Dipartimento di Matematica e Applicazioni \\ Via Cozzi 55, 20126 Milano - ITALY}
\email{stefano.pigola@unimib.it}
\author{Alberto G. Setti}
\address{Universit\`a dell'Insubria \\ Dipartimento di Scienza e Alta Tecnologia - Sezione di Matematica \\ via Valleggio 11\\
I-22100 Como, ITALY}
\email{alberto.setti@uninsubria.it}

\date{\today}

\begin{abstract}
The interplay between a weighted volume measure and its naturally associated weighted Laplace operator is fundamental to the geometric analysis of weighted structures governed by gradient vector fields. Beyond the gradient setting, however, the theory has remained considerably less developed, largely because this natural correspondence is lost. In this paper, we identify invariant measures as the appropriate framework for extending the theory to general drifted Laplace operators. This approach yields a systematic treatment of parabolicity in the non-gradient setting, reveals its connection with volume growth, and leads to new volume finiteness results for non-gradient geometric structures.
\end{abstract}

\maketitle

\section{Introduction}

 Let $(M,\metric,\nu)$ be a smooth metric measure space, where $(M,\metric)$ is an $m$-dimensional Riemannian manifold and $\nu$ is a positive Radon measure. Many important examples arising in geometry are modeled by measures induced by smooth densities, namely $\nu = \vol_f \doteq e^{-f}\vol$. This class, called weighted manifolds, has been extensively studied from both the analytic and geometric viewpoints (cf. e.g. \cite{AMR,BPS,Cas,CaMa,CGY,Gr3,Gr4,IRS,Lo,Mo1,Mo2,MW1, MW2,PeWi,PRRS,WW} just to touch some aspects of the immense literature on the subject), and contains important geometric structures such as gradient Ricci solitons \cite{Ha}, and self-shrinkers of the mean curvature flow (MCF), \cite{CM}. 

From a PDE perspective, the weighted measure $\vol_f$ naturally gives rise to the second-order drifted operator $\Delta_f \doteq \Delta - \langle \nabla f, \nabla\,\rangle$, called $f$-Laplace-Beltrami operator, which is symmetric with respect to $\vol_f$, and reveals deep geometric and potential-theoretic properties of the weighted manifold. For instance, by now standard weighted comparison arguments show that every complete gradient shrinking Ricci soliton has finite weighted volume \cite{WW}. In the same vein, a beautiful result of Q. Ding and Y.L. Xin, \cite{DX} states that every properly immersed self-shrinker has finite weighted volume. A common viewpoint to investigate the finiteness of the weighted volumes of gradient shrinking Ricci solitons and self-shrinkers of the MCF was established by X. Cheng and D. Zhou, \cite{CZ}.

The interaction between the $f$-Laplacian and the weighted measure $\vol_f$ provides a natural framework for the geometric analysis of a weighted manifold, underlying the geometric aspects of the theory through the Bakry--Émery Ricci tensors, comparison theory, and its potential-theoretic aspects, including parabolicity via volume growth criteria, capacity estimates, global Green formulas, and maximum principles (see again \cite{AMR,Gr3,WW}).

More generally, when dealing with geometric structures outside the gradient setting, like non-gradient Ricci solitons, one is naturally led to consider a general drifted operator of the form $\Delta_X u \doteq \Delta u - \langle X,\nabla u\rangle$, called $X$-Laplacian, where $X$ is an arbitrary vector field. In this broader context, much less is known, mainly due to the fact that the $X$-Laplacian does not arise from a quadratic form with respect to a natural distinguished measure.

In this paper, we borrow a probability perspective and look for a weighted measure naturally associated with the given drifted Laplacian. A major role in understanding how much of the previous picture can be extended to this general setting is played by the concept of invariant measures. As we shall see, this change of perspective is crucial in discovering new phenomena in the potential theory of the $X$-Laplacian. In particular, we address the fundamental question of whether there exist canonical measures naturally associated with the $X$-Laplacian and, if so, to what extent the corresponding notion of volume is still related to parabolicity in this broader setting.

\subsection{Main results} We briefly introduce some notation before stating the main results. The reader can find the missing definitions in the subsequent sections.

\subsection{Finiteness results for invariant measures}
Let $M_f\doteq (M,\metric,\vol_f)$ be a weighted manifold, where $\vol_f = e^{-f}\vol$ is the weighted volume measure that is absolutely continuous with respect to the Riemannian measure $\vol$. The function $f \in C^\infty(M)$ is referred to as the potential function of the weighted manifold. The $f$-Laplacian operator acts on $C^2$-functions by
\[
\Delta_f u = \Delta u - \< \nabla f , \nabla u\>.
\]
Equivalently, letting $\div_fX = e^{f} \div(e^{-f}X)$ be the  weighted divergence, we have
\[
\Delta_f u = \div_f(\nabla u).
\]
The $f$-divergence theorem
\[
\int_\Omega \div_f X \dv_f = \int_{\partial \Omega}\< X, \eta\> d\s_f
\]
valid for a smooth vector field $X$ defined on a bounded open subset $\Omega \subset M$ with smooth boundary and outward unit normal vector field $\eta$, readily yields that $\Delta_f$ is symmetric on $C^\infty_c(M)\subset L^2(\vol_f)$ (and, in fact, essentially self-adjoint if $M$ is complete). 

The geometric analysis of $M_f$ is related to a family of tensors collectively referred to as Bakry-Emery Ricci tensors. These are defined by
\[
\ric_{f,n} = \ric + \Hess f - \frac{1}{n-m}df \otimes df,
\]
where $n>m$ is the virtual dimension. In case $n = +\infty$ one simply writes $\ric_f$. The class of weighted manifolds contains the following important  examples arising from geometry.

\begin{example}[Gradient Ricci solitons]
    We say that $M_f$ is a gradient Ricci soliton with soliton constant $\lambda \in \rr$ if $\ric_f = \lambda \metric$. The soliton is shrinking, steady or expanding if, respectively, $\lambda >0$ $\lambda=0$ or  $\lambda<0$. The prototypical example of a gradient shrinking Ricci soliton is the Gaussian space $\CG^{m}= \rr^m_f$ with $f(x)= |x|^2/2$. This clearly has finite volume $\vol_f(\CG^m)<+\infty$. In fact, as a consequence of weighted comparison arguments (see \cite{WW}), we know that any shrinking Ricci soliton $M_f$ has finite weighted volume $\vol_f(M_f)<+\infty$.
\end{example}

\begin{example}[Self-shrinkers of the MCF]
    Let $\vec x: M^m \to \rr^{m+1}$ be an isometrically immersed hypersurface. We say that $M$ is a self-shrinker of the Mean Curvature Flow if the mean curvature vector field of the immersion satisfies
    \begin{equation}\label{eq:MCF1}
    \vec x ^\perp = - \vec H,     
    \end{equation}
    where $(\cdot)^\perp$ denotes the projection on the normal bundle.  The hypersurface \(M\) inherits the ambient Gaussian weight \(f =|\vec x|^2/2\), thus defining a weighted manifold \(M_f\). Since \(\nabla^{\rr^{m+1}} f (\vec x)=\vec x\), the self-shrinker equation \eqref{eq:MCF1} can be rewritten as
\[
\vec H_f:=\vec H+(\nabla^{\rr^{m+1}} f )^\perp=0,
\]
showing that self-shrinkers are precisely weighted minimal hypersurfaces in Gaussian space. 
Concrete examples of self-shrinkers are the Gaussian space $\CG^m$, the sphere $\ss^{m}_{\sqrt{m}}$ and the generalized cylinder $\ss^{k}_{\sqrt{k}} \times \rr^{m-k}$. According to T. Colding and W. Minicozzi, \cite{CM}, these are the only mean-convex properly embedded self-shrinkers. Clearly, all these examples have finite weighted volume. A beautiful result of Ding and Xin \cite{DX}, later completed and extended by Cheng and Zhou \cite{CZ}, asserts that every properly immersed self-shrinker $M_f$ has finite weighted volume. The proof has a genuinely PDE flavor, as no comparison argument is available in the absence of curvature bounds.
\end{example}

As previously noted, a natural question is whether there is a unified perspective from which one can deduce the finiteness of the weighted volume of shrinking Ricci solitons and self-shrinkers for the MCF. In the remarkable work \cite{CZ}, the authors detected a set of differential conditions on the potential function (automatically satisfied in the geometric examples) under which a weighted manifold has finite weighted volume and polynomial Riemannian volume growth. 
Restricting our attention to the weighted setting their result states
\begin{theorem}[Cheng-Zhou]\label{th:CZ}
    Let $M^{m}_{f}$ be a complete weighted manifold whose potential function satisfies the following conditions:
\begin{itemize}
 \item [(a)] $0 \leq f(x) \to +\infty$, as $x \to \infty$;
 \item [(b)] $|\nabla f|^{2} \leq f$ on $M$;
 \item [(c)] $\Delta_{f}f +f \leq \kappa$ on $M$,
\end{itemize}
for some constant $\kappa \in \rr$. Then  $\vol_f(M_f) <+ \infty$.
\end{theorem}
Reversing the usual PDE perspective, let us now consider an invariant measure for $\Delta_f$ defined in the weighted manifold $M_f$, that is, an absolutely continuous measure $d\mu = u\, d\vol_f$ whose positive smooth density $u$ satisfies 
\[
\Delta_f^\ast u = 0\, \text{ on }M,
\]
where $\Delta_f^\ast$ denotes the formal adjoint of $\Delta_f$ in $L^2(d\vol_f)$. By symmetry, $\Delta_f^\ast$ coincides with $\Delta_f$ and, in particular, the background measure $\dv_f = 1 \dv_f$ is itself invariant for $\Delta_f$. Fundamental contributions to the invariant measure theory were made by Khasminskii in the context of semigroups and stochastic processes, \cite{Kh}. Function theory again plays a key role, as emphasized by the following general result (see also \cite{BR}).

\begin{theorem}[Khasminskii]\label{th:invariant-Khasm}
    Assume that $M_f$ supports a $C^2$-Lyapunov function $\gamma \colon  M \to \rr$, namely, $\gamma$ satisfies the following conditions:
    \begin{itemize}
        \item [(a)] $0 \leq \gamma(x) \to +\infty$ as $x \to \infty$;
        \item [(b)] $\Delta_f \gamma(x) \to -\infty$ as $x \to \infty$.
    \end{itemize}
Then there exists an invariant measure $\mu$ for $\Delta_f$ of finite mass, $\mu(M) <+\infty$.
\end{theorem}

If $M_f$ is parabolic, namely,  solutions $0 \leq w \in C(M) \cap W^{1,2}_{loc}(M)$  of $\Delta_f w \leq 0$ are constant, then by symmetry also positive solutions of the adjoint equation $\Delta_f^\ast u = 0$ on $M$ must be constant. Therefore, $\vol_f$ is, up to multiplicative constants, the unique invariant measure with respect to $\Delta_f$. On the other hand, it is apparent that Lyapunov functions meet the Khasminskii test of parabolicity:
there exists a function $\gamma \in \text{\rm Lip}_{loc}(M)$ satisfying:
\begin{itemize}
\item [(i)] $\g(x) \to \infty$ as $x \to \infty$;
\item [(ii)] $\Delta_{f} \g \leq 0$ outside a compact set.
\end{itemize}
This leads straightforwardly to a generalization of Theorem~\ref{th:CZ}, showing that  the finiteness of the weighted volume can be obtained without imposing the gradient condition and assuming that $f$ satisfies a substantially relaxed  PDE.

\begin{theoremA}[Cheng-Zhou revisited]\label{th:CZ_1}
Every complete weighted manifold $M^{m}_{f}$ whose potential function satisfies:
\begin{itemize}
 \item [(a)] $0 \leq f(x) \to +\infty$, as $x \to \infty$;
 \item [(b)] $\Delta_{f}f \to -\infty$, as $x \to \infty$,
\end{itemize}
has finite weighted volume $\vol_f(M_f) <+ \infty$.
\end{theoremA}

It is remarkable that, even in the symmetric setting of weighted manifolds, Liouville-type results, such as parabolicity, play a central role in establishing the finiteness of the weighted volume measure. Moving to the non-symmetric setting, the preceding discussion suggests that  invariant measures may be effective  in extending the symmetric picture. Indeed, as we shall see shortly, we are able to obtain a version of Theorem~\ref{th:CZ_1} for the $X$-Laplacian.


Let $X$ be a smooth vector field and let the 
non-symmetric $X$-Laplacian $\Delta_X$ be the operator which acts on $C^2$-functions by
\[
\Delta_X u = \Delta u - \< X , \nabla u\>.
\]
An absolutely continuous measure $d\mu = u\, d\vol_f$ is said to be invariant for $\Delta_X$ if its positive smooth density $u$ satisfies 
\[
\Delta_X^\ast u = 0\, \text{ on }M,
\]
where $\Delta_X^\ast$ denotes the formal adjoint of $\Delta_X$ in $L^2(\dv)$.

First of all, it is natural to ask whether the available comparison theory for the $X$-Laplacian (cf. \cite{BQ,CMMR}) yields  volume comparison results for its invariant measures. The following result, proved in Section \ref{section:geometry}, provides a positive answer to this question.\footnote{See \cite{BQ} for an 
instructive discussion.} Recall that, associated with a given smooth vector field $X$, there is a family of Ricci tensors defined, for each virtual dimension $n \in (m,+\infty]$, by the formula
\[
\ric_{X,n}=\ric+\frac{1}{2}L_X \metric-\frac{1}{n-m}X^\flat\otimes X^\flat,
\]
where $L_X$ denotes the Lie-derivative along $X$. We simply write $\ric_X$ when $n=+\infty$.

\begin{theoremA}\label{th:volumecomparison_intro}
Let $(M,\metric )$ be a complete, $m$-dimensional Riemannian manifold endowed with a smooth vector field $X$. Let $\rho(x)$ be the intrinsic geodesic distance function from a reference origin $o \in M$. Assume that either one of the following conditions is satisfied:
\begin{itemize}
    \item [(i)] $ \ric_{X,n} \geq (n-1) K$ for some constant $K \in \rr$ and some virtual dimension $n>m$.\medskip
    \item [(ii)] $ \ric_{X} \geq (m-1)K$ for some constant $K > 0$.
\end{itemize}
Then, correspondingly, the following volume estimates hold for every $X$-invariant measure $d\mu=e^{-f}\dv$:
\begin{itemize}
    \item [(j)] The function $R \mapsto \vol_f(B_R(o))/V_K(R)$ is non-increasing for a.e. $R>1$, where $V_K(R)$ denote the volume of the geodesic ball of radius $R$ in the simply connected $n$-dimensional space form of constant sectional curvature $K$.\medskip
    \item [(jj)] $\mu(M) < +\infty$.
\end{itemize}
\end{theoremA}

\subsection{Parabolicity: weighted volumes and other criteria}

A useful tool for studying the global geometry of a space is represented by the potential theoretic notion of parabolicity. We say that a manifold $M$ is $X$-parabolic, or parabolic with respect to $\Delta_X$, if the following Liouville property is verified: a positive function $u \in C^0(M) \cap W^{1,2}_{loc}(M)$ satisfying $\Delta_X u\leq 0$ weakly on $M$ must be constant. In the symmetric setting where $X=\nabla f$, $f$-parabolicity is a well-established concept with several equivalent characterizations in terms of Brownian motion, Green function, capacity, Ahlfors maximum principle, Hardy-type inequality and Khasminskii test, to mention only a few (see e.g. \cite{AMR,Gr3}). For general drifts, this picture is less developed, and only a few equivalences are available. For instance, combining results  from \cite{CMMR} and \cite{MP,MP1},  $X$-parabolicity can be characterized in terms of the Ahlfors maximum principle and the following Khasminskii test: there exists a function $\gamma \in C^2(M)$ satisfying:
\begin{itemize}
\item [(i)] $\g(x) \to \infty$ as $x \to \infty$;
\item [(ii)] $\Delta_{X} \g \leq 0$ outside a compact set.
\end{itemize}
Another relevant characterization of parabolicity is given by the validity of a global Green formula for a suitable class of functions (cf. \cite{Gr1,GM}).

One of the main goals of the present paper is to complete the picture by extending the classical characterizations of parabolicity to the setting of general drifted Laplace operators. Again, invariant measures provide the key tool in this investigation. To this end, we first define a non-variational capacity associated with the drifted Laplacian via its invariant measure, which is similar to the Laplacian capacity; see Definition \ref{def:X-capacity} (cf. \cite{Ba,BMP,DKK}): given a compact set $K \subset M$, define
$$\CAP_{X,\mu}(K) = \inf\left\lbrace \int_M \vert \Delta_{X} \varphi \vert d\mu \colon \varphi \in C^\infty_{c}(M), \  \varphi \geq 1 \ \text{in some neighborhood of }  K\right\rbrace.$$
Extending \cite{Gr1,GM}, we also introduce the following $X$-Laplacian version of the global Green formula, which we refer to as the $X$-Green formula, see Definition \ref{def:X-GreenFormula}: for every function $w \in L^\infty(M)$ such that $\Delta_X w \in L^1_{\text{loc}}(M,\mu)$ and $\left(\Delta_{X} w\right)_+ \in L^1(M,\mu)$ it holds
$$ \int_M \Delta_{X} w \, d\mu = 0.$$

One of our main results is the following characterization of $X$-parabolicity which also involves the non-symmetric Green function of the $X$-Laplacian.

\begin{theoremA}\label{thm_equiv_Xparab_intro}
Let $(M,\metric)$ be a Riemannian manifold endowed with an $X$-invariant measure $\mu$. The following are equivalent.
\begin{itemize}
\item[1.] $M$ is $X$-parabolic.
\item[2.] There is no positive $X$-Green function, i.e., $g_X(x,y) \equiv \infty$.
\item[3.] $\CAP_{X,\mu}(K)=0$ for every compact set $K$. 
\item[4.] $(M,\metric,\mu)$ satisfies the $X$-Green formula.
\end{itemize}
\end{theoremA}

\begin{remark}\label{rmk_brezis-marcus-ponce-equiv}
As a consequence of Theorem~\ref{thm_equiv_Xparab_intro} we have the equivalence between the vanishing of the classical Newton capacity and the vanishing of the Laplacian capacity for $\Delta_f$, 
a problem first investigated for $\Delta$ in the Euclidean setting in \cite{BMP} and in the Riemannian setting in \cite{GPS}; see Corollary \ref{cor:equivalence} below.
\end{remark}

As mentioned above, the characterization of parabolicity in terms of the Liouville property plays a fundamental role in the uniqueness of invariant measures. Indeed, if $\mu$ is an invariant measure for $\Delta_X$ and $M$ is $X$-parabolic, then $M$ is also parabolic for its dual operator with respect to $\mu$ (see Proposition \ref{equiv_green_Xparab}) and every invariant measure for $\Delta_X$ must be a constant multiple of $\mu$. Therefore, since the Khasminskii existence result of Theorem \ref{th:invariant-Khasm} works for $\Delta_X$, \cite{Kh,BR, BRW}, repeating verbatim the reasoning we used in the weighted setting, we get the announced extension of Theorem \ref{th:CZ_1} to  non-symmetric operators.

\begin{theoremA}
  Let $(M,\metric)$ be a Riemannian manifold endowed with an $X$-invariant measure $\mu$. If there exists a function $\gamma \in C^2(M)$ satisfying
\begin{itemize}
 \item [(a)] $0 \leq \gamma(x) \to +\infty$, as $x \to \infty$;
 \item [(b)] $\Delta_{X}\gamma \to -\infty$, as $x \to \infty$,
\end{itemize}
then $\mu$ is the essentially unique $X$-invariant measure and $\mu(M)<+\infty$.
\end{theoremA}

In the setting of weighted manifolds, the measure $\vol_f$ is naturally the unique invariant measure for $\Delta_f$ with respect to which it is symmetric (cf. Lemma~\ref{lem_symm_vol_f}). This fact identifies $\vol_f$ as a canonical measure for $\Delta_f$ and underlies the following classical interplay between $f$-parabolicity and weighted volume growth on complete manifolds:
  \[
\int^{+\infty} \frac{dR}{\H^{m-1}_{f}(\partial B_R(o))} = +\infty \qquad \Rightarrow \qquad f\text{-parabolicity},
    \]
where $\H^{m-1}_f = e^{-f} \H^{m-1}$ denotes the weighted $(m-1)$-dimensional Hausdorff measures.

The challenge in understanding whether, and to what extent, volume growth is related to $X$-parabolicity stems from the absence of a distinguished measure naturally associated with $\Delta_X$. Moreover, one can construct examples of non-parabolic complete manifolds carrying invariant measures of moderate mass-growth (see Example \ref{ex-1}). Surprisingly, this fact is already visible at the level of the $f$-Laplacian.

We address the problem of identifying which invariant measures, if any,  play a canonical role for the $X$-parabolicity in Section~\ref{section:volume} where we propose the following definition, see Definition \ref{def:L2defect}:
given a smooth $X$-invariant measure $d\mu=e^{-f} \dv$, the $L^2$-defect for $\Delta_X$ to be symmetric with respect to $\mu$ is the number 
    \[
    \d_f(\Delta_X) = \| X - \nabla f \|_{L^2(M,\dv_f)} = \left\{ \int_M |X - \nabla f |^2 \dv_f \right\}^{\frac{1}{2}}.
    \]
We call $\mu$ a {\it canonical measure} for $\Delta_X$ if $\d_f(\Delta_X)<+\infty$.

It is worth noting that a type of $L^p$-finiteness condition has already appeared in the literature on uniqueness problems of probability invariant measures associated with semigroups, \cite{ABR}. However, this condition emerged independently in our work on invariant measures of infinite mass, where it acquires a completely new meaning motivated by the search for novel links with potential theory.  In fact, with this new concept at hand, we can establish the crucial connection between $X$-parabolicity and moderate volume growth.

\begin{theoremA}\label{thm_X-parab_intro}
Let $(M,\metric )$ be a complete Riemannian manifold and let $d\mu= e^{-f}\dv$ be a canonical measure for the $X$-Laplacian. If, for some origin $o \in M$,
    \[
    \int^{+\infty} \frac{dR}{\H^{m-1}_f (\partial B_R(o))} = +\infty,
    \]
then $M$ is parabolic with respect to both $\Delta_X$ and its dual (with respect to $\mu$). In particular, $\mu$ is the unique invariant measure for the $X$-Laplacian.
\end{theoremA}

\begin{remark}
The proof of Theorem~\ref{thm_X-parab_intro} follows by showing that, under the canonical measure assumption, the vanishing of the $f$-capacity implies the vanishing of the $X$-capacity. As we will see in Proposition~\ref{controllo_fCap_XCap}, for every $X$-invariant measure, we have the estimate
$$
\CAP_f(K,U)\leq \frac12 \CAP_X(K,U).
$$
\end{remark}

As a by-product of this original insight, we also investigate the extension of the Liouville-type result proved by S.T. Yau in the `70s which states that on any complete Riemannian manifold, nonnegative subharmonic functions are necessarily constant if they are in some $L^p$-class, $1 < p < +\infty$.
\begin{theoremA}\label{Yau_Lp_intro}
Let $(M,\metric)$ be a complete Riemannian manifold endowed with an $X$-invariant measure $d\mu = e^{-f}\dv$, and assume that
the $L^\infty$-defect given by 
\[ 
\d^\infty_f(\Delta_X) = \| X - \nabla f \|_{L^\infty(M,\dv_f)}
\]
is finite.
Then, for every $v \in C^{2}(M)$ and $1<p<\infty$,
\[
\begin{cases}
\Delta_X v \geq 0 \\
v \geq 0 \\
v \in L^{p}(M,\mu)
\end{cases}
\Longrightarrow
v \equiv const.
\]
\end{theoremA}
As we shall see in Section \ref{section:Yau}, the finiteness of the $L^\infty$-defect can be sharpened to a suitable growth condition, which cannot be completely avoided, see Example \ref{sharpness_yau}.

\subsection{Organization of the paper} 
In Section \ref{section:invariantmeasures} we review basic aspects of invariant measures. A regularization technique via Riemannian convolution kernels that extends the construction in \cite{PVV} to general non-symmetric operators and will be a fundamental tool for our analysis,  is developed in Section \ref{section:kernels}.  In Section \ref{section:X-parab}, after a brief overview of the currently known theory of $X$-parabolicity, we introduce new characterizations in terms of a suitable notion of capacity and of the validity of certain Green formulas. This is the most technical part of the paper. 
Section \ref{section:volume} is devoted to the investigation of volume growth conditions. To this end, we introduce a new notion  of canonical invariant measure and explore some of its applications. In Section \ref{section:Yau} we show to what extent canonical measures are natural by generalizing, in the setting of drifted Laplacians, the celebrated Liouville theorem by S.T. Yau on $L^p$, positive subharmonic functions. Finally, in Section \ref{section:geometry}, we apply the previous theory to the geometry 
of generic shrinking Ricci solitons. This  requires 
new volume estimates under curvature bounds that are derived here.

\section{Aspects of invariant measures}\label{section:invariantmeasures}

\subsection{Generalities on invariant measures}

Let $(M,\metric ,\nu)$ be a Riemannian manifold endowed with a positive Radon measure $\nu$. Throughout this paper, we will always assume that $d\nu= w  \dv$ for a  density  $0<w \in C^\infty(M)$.

Let $\L$ be a second order, elliptic, differential operator with smooth coefficients on $M$. Its formal adjoint, with respect to $\nu$, is denoted by $\L^{\ast,\nu}$. By definition, for every $\vp,\psi \in C^{\infty}_{c}(M)$,
\[
\int_{M} \psi \L\vp\, d\nu = \int_{M} \vp \L^{\ast,\nu}\psi\, d\nu.
\]
Note that $\L^{\ast,\nu}$ is still elliptic with smooth coefficients.

\begin{example}
 Assume that
\begin{itemize}
 \item  $\nu = \vol$ is the Riemannian measure;
 \item $\Delta_{X} = \Delta - X$ where $X$ is a smooth vector field on $M$.
\end{itemize}
It is customary to call $\Delta_{X}$ the $X$-Laplacian. Note that, if $X = \nabla f$, then $\Delta_{X}$ is nothing but the $f$-Laplacian
\[
\Delta_{f} u = \Delta u -\<\nabla u, \nabla f\> = e^f \div(e^{-f}\nabla u)
\]
associated with the weighted Riemannian manifold
\[
M_{f} = (M,\metric ,\vol_f),
\]
where $d\vol_f=e^{-f} d\vol$.  Integrating by parts, it is easy to check that
\[
\Delta^\ast_X u \doteq \Delta_{X}^{\ast,\vol} u = \Delta u + \langle X, \nabla u \rangle + u\div X,
\]
which is a special case of the Fokker-Planck-Kolmogorov operator. Note that $\Delta^{\ast}_X$ can be written in divergence form as
\[
\Delta^\ast_X u =   \div (\nabla u + uX)=  \div \A(x,u,\nabla u)
\]
where $\A \colon M\times \mathbb{R}\times TM \rightarrow \mathbb{R}$ is given by $\A(x,r,\xi) = \xi + rX$. 
\end{example}

\begin{definition}
 We say that a positive Radon measure $\mu$ on $(M,\metric ,\nu)$ is an invariant measure for an elliptic operator $\L$ if $\mu$ is a distributional solution to  the equation
 \[
 \L^{\ast,\nu} \mu =0.
 \]
 Explicitly, the measure $\mu$ is invariant for $\L$ if
  \[
 \int_{M} \L \vp \, d\mu =0
 \]
 holds for every $\vp \in C^{\infty}_{c}(M)$.
\end{definition}

For a general elliptic operator, the existence of invariant measures is a consequence of the following simple abstract result.

\begin{lemma}\label{lemma-abstractexistence}
With the notation introduced above:
\begin{enumerate}
 \item [(a)] Assume that there exists a smooth solution $ u  > 0$ to the equation
\begin{equation}\label{ad-equation}
\L^{\ast,\nu} u = 0.
\end{equation}
Then,  $d\mu = u \, d\nu$ is a positive invariant measure for $\L$.
\item [(b)] Conversely, let $\mu$ be a positive invariant measure for $\L$. Then, $d\mu = u  d\nu$ for some  function $0< u \in C^\infty(M)$ satisfying \eqref{ad-equation}.
\item [(c)] An invariant measure $\mu$ for $\L$ is finite (hence a probability measure up to renormalization) if and only if $d\mu = u d\nu$, where $u>0$ is a smooth solution of \eqref{ad-equation} satisfying $u \in L^{1}(M,\nu)$.
\end{enumerate}
\end{lemma}
 
{\begin{proof}
We include a quick proof for the sake of completeness.\smallskip

\noindent(a)  Take $\vp \in C^{\infty}_{c}(M)$ and note that, since $u$ is smooth, then by the theory of distributions
 \[
 \int_{M} \L\vp \, d\mu = \int_{M} u \L \vp \,  d\nu = \int_{M}\vp \L^{\ast,\nu} u \, d\nu.
  \]
Therefore, $d\mu = u \, d\nu$ is invariant for $\L$ if and only if \eqref{ad-equation} is satisfied.\smallskip

\noindent (b) Let the positive Radon measure $\mu$ be invariant for $\L$. Since $\L^{\ast,\nu}$ is elliptic with smooth coefficients, by (Weyl) elliptic regularity, the distributional solution $\mu$ of $L^{\ast,\nu}\mu=0$ must be a smooth function, namely, $d\mu=ud\nu$ for some smooth density $u>0$. It then follows that $u$ solves \eqref{ad-equation}.\smallskip

\noindent (c) It is a trivial consequence of (a) and (b).
\end{proof}

\begin{example}[Symmetry and invariant measures]
If $\mu$ is a measure on $M$ with respect to which an operator $\L$ is symmetric and $\L 1 =0$, then $\mu$ is invariant for $\L$. For instance, the Riemannian measure $\vol$ is invariant for the Laplace-Beltrami operator $\Delta$. The {\it Gaussian-like} measure $d\vol_f = e^{-f}d\vol$ is invariant for the $f$-Laplacian $\Delta_{f}$.
\end{example}

\begin{example}[$X$-invariant measures]\label{example:X-inv}
On the Riemannian manifold $(M,\metric,\vol)$, let us consider the $X$-Laplacian $\L=\Delta_X$. We say that the measure $\mu$ is $X$-invariant if it is an invariant measure for $\Delta_X$. In this case, by Lemma \ref{lemma-abstractexistence}, $d\mu= u \dv$ for a smooth function $u>0$. We write $u=e^{-f}$ so that $\mu = \vol_f$. The invariance condition therefore reads, for every $\vp \in C^\infty_c(M)$,
\begin{align*}
    0 &= \int_M\Delta_X \vp \dv_f \\
    &= \int_M \left( \Delta \vp - \<X, \nabla \vp\>\right)\dv_f \\
    &= \int_M  \Delta_f \vp \dv_f - \int_M\<X-\nabla f,\nabla \vp\>\dv_f \\
    &= \int_M  \vp \div_f(X-\nabla f)\dv_f.
\end{align*}
It follows that
\begin{equation}\label{eq:X-inv}
\mu=\vol_f \text{ is $X$-invariant} \quad \Longleftrightarrow \quad \div_f(X-\nabla f) =0.    
\end{equation}
It is worth noting that, from this characterization and the $f$-divergence theorem, one can easily deduce the explicit expression of the formal adjoint $\Delta_X^{\ast, \mu}$. In fact, we have
\[
\Delta_X^{\ast, \mu} = \Delta_{X_\mu}
\]
with
\[
X_\mu = -X + 2 \nabla f.
\]
To see this, simply compute, for every $\vp,\psi \in C^\infty_c(M)$,
\begin{align*}
    \int_X \psi \Delta_X \vp \dv_f &= \int_M \psi \Delta_f \vp \dv_f - \int_M \< X - \nabla f, \nabla \vp \> \psi \dv_f \\
    &= \int_M \vp \Delta_f \psi \dv_f + \int_M \vp \div_f( \psi (X- \nabla f)) \dv_f \\
    &= \int \vp( \Delta\psi -\<\nabla \psi, \nabla f\>)\dv_f
    +\int_M \vp \<\nabla \psi,X-\nabla f\> \dv_f \\
    &= \int_M \vp ( \Delta \psi-\< \nabla \psi,-X + 2\nabla f\>)\dv_f \\
    &= \int_M \vp \Delta_{X_\mu} \psi \dv_f,
 \end{align*}
 as claimed.
\end{example}

If $\L$ generates a contractive semigroup $P_{t}$, then the notion of invariant probability measure can be introduced in terms of  this semigroup. Notice that this is the case for the $X$-Laplacian (see \cite{MPW}, Section 4). Let $C_{b}(M)$ denote the space of  bounded continuous functions.

\begin{definition}[Invariant probability measures for semigroups]
We say that a probability measure $\mu$ is invariant for $P_{t}$ if $P_{t}^\ast\mu = \mu$, namely, for every $\vp \in C_{b}(M)$,
\[
\int_{M}P_{t}\vp \, d\mu = \int_{M} \vp\, d\mu.
\]
\end{definition}

The semigroup viewpoint on  invariant probability measures gives interesting connections with the {\it conservative property}, also called {\it stochastic completeness} or {\it Markovian property}.

\begin{definition}[Conservativeness]
 The semigroup $P_{t}$ is conservative (equivalently, $P_{t}$ is stochastically complete/Markovian) if
 \[
P_{t} 1 = 1,\quad \forall t > 0.
\]
\end{definition}

Applying the definition of invariant probability measure with $\vp \equiv 1$ and using the sub-Markovian property of $P_t$, we obtain the following

\begin{lemma}\label{lemma-stochcompl}
If an invariant probability measure for the semigroup $P_t$ generated by $\L$ exists, then it is conservative. 
\end{lemma}

It is well known that in the case of the $f$-Laplacian, the conservative property has an elliptic counterpart in terms of a Liouville property for bounded nonnegative solutions to $\Delta_f v \geq \lambda v$. For general elliptic operators, it is not obvious that this connection still survives.
On the other hand, as soon as the generator of the semigroup satisfies the above mentioned Liouville property, then the elliptic and the semigroup viewpoints on invariant measures coincide. To make everything precise we record the following

\begin{definition}[$\l$-Liouville]
We say that the elliptic operator $\L$ satisfies the $\l$-Liouville property (it is $\l$-Liouville for short) if for any $\l>0$ the only nonnegative, bounded solution of the problem
\[
\L u = \l u\quad  \text{on } \ M
\]
is a constant function.
\end{definition}

The following result can be culled from \cite[Proposition 5.9, Corollary 5.10]{MPW}.

\begin{proposition}[Conservativeness = $\l$-Liouville]
$P_{t}$ has the conservative property if and only if $\L$ is $\l$-Liouville for some, hence any, $\l >0$.
\end{proposition}

With this preparation, we have the following result (see \cite[Lemma 6.1]{MPW}).

\begin{proposition}[Elliptic = semigroup invariant measures]\label{prop-abstractexistence}
Let $\mu$ be a finite measure. Then the following hold.
\begin{itemize}
 \item [a)] If $\mu$ is an invariant probability measure for $P_{t}$, then $\L$ is $\l$-Liouville and $\mu$ is an invariant measure for $\L$.
 \item [b)] If $\mu$ is an invariant probability measure for $\L$ and $\L$ is $\l$-Liouville, then $P_{t}$ is conservative and $\mu$ is an invariant measure for $P_{t}$.
\end{itemize}
\end{proposition}

\begin{remark}
Concerning b) it is possible to show that, without $\l$-Liouville, the implication does not hold. Indeed, it is proved in \cite{Sh} that, for suitable choices of the vector field $X$, there are many invariant probability measures for $\Delta_{X}$, in contrast to Doob uniqueness result (see Theorem~\ref{th-Doob} below). In fact, elaborating on \cite{BRS}, Shaposhnikov \cite{Sh} introduces a machinery to produce invariant probability measures. See also Example \ref{sharpness_yau} below.
\end{remark}

\subsection{Existence of invariant measures for the $X$-Laplacian}

The existence of invariant measures for general elliptic operators is related to the existence of positive solutions to the equation \eqref{ad-equation} governed by the dual operator, as shown in Lemma~\ref{lemma-abstractexistence}. Abstractly, for general divergence-type operators, the latter problem can be investigated in terms of the validity of the (weak) maximum principle (see \cite[Corollary~2.2.3]{BKRS_Book}). Given a smooth vector field $X$ defined on a manifold $M$, the existence of an $X$-invariant measure boils down to the existence of a global positive solution to the Fokker-Planck-Kolmogorov equation.

\begin{theorem}\label{existence_inv_measure}
Let $M$ be a Riemannian manifold and $X$ be a smooth vector field on $M$. Then, there is a  solution $0<u \in C^\infty(M)$ to the equation
\begin{equation}\label{eq:FP}
\div(\nabla u + uX) = 0.    
\end{equation}
\end{theorem}
\begin{proof} In the case where $M=\mathbb{R}^m$ the conclusion follows from \cite[Corollary 2.2.3]{BKRS_Book}. Since the arguments therein are local, they extend to manifolds  using coordinate charts and partitions of unity. We provide a detailed proof for  completeness and the convenience of the reader. 

We begin proving the following
\begin{lemma}\label{lemma:ExDirPb}Let $U$ be a relatively compact open set with smooth boundary. Then there exists a unique function 
$w\in C^\infty(\overline{U})$ such that 
\begin{equation}\label{eq_u-in-U}
\Delta^*_Xw =\div(\nabla w + w X) = 0, \ \ w>0 \ \text{in } \ U \quad \text{and} \quad w\equiv 1 \ \ \text{on} \ \partial U. 
\end{equation}
\end{lemma}
Indeed, set $w=z+1$ so that \eqref{eq_u-in-U} is equivalent to $\Delta_X^*z=-\Delta_X^* 1 = -\div X$ in $U$ and $z=0$ on $\partial U$. By the Fredholm alternative for the Dirichlet problem \cite[Theorem~6.15]{GT}, this problem is uniquely solvable in $C^{2,\alpha}(\overline U)$ for every right-hand side provided the homogeneous problem has only the trivial solution; smoothness up to the boundary then follows from  Schauder regularity, since all data are smooth.

Thus, suppose  $z\in C^{2,\alpha}(\overline\Omega)$ satisfies $\Delta_X^*z=0$ in $\Omega$ and $z=0$ on $\partial\Omega$. To show that $z\equiv 0$,   fix $f\in C_c^\infty(\Omega)$ and let $v$ be the
solution of
\begin{equation}\label{eq:v}
 -\Delta_X v=f\ \text{in }\Omega,\qquad v=0\ \text{on }\partial\Omega ,
\end{equation}
which exists by \cite[Theorem~6.14]{GT} because $\Delta_X$ has no zero-order term. 
Integrating the identity  
\begin{equation}\label{eq:green} 
z\Delta_Xv -v\Delta_X^*z=\div( z\nabla v-v\nabla z -vzX),
\end{equation} 
applying the divergence theorem and noting that the  boundary terms vanish
since $v=z=0$ on $\partial U$, yield
\[
 -\int_U z\,fd\vol=\int_U z\,\Delta_X vd\vol=\int_U v\,\Delta_X^*zd\vol=0 .
\]
As $f$ was arbitrary, $z\equiv0$, as required.

To show that $w\geq 0$ in $\overline{U}$, let $f\in C_c^\infty(U)$ with $f\geq0$ and let $v$ solve \eqref{eq:v}.  Since $\Delta_Xv=-f\leq0$ and $v=0$ on $\partial U$, 
by the weak maximum principle, which holds since $\Delta_X$ has no zero-order term, 
$v\geq0$ in $U$ 
and therefore $\partial_n v\leq0$ on
$\partial U$.  

Integrating \eqref{eq:green} with $w$ instead of $z$ and using
$\Delta_X^*w=0$, $v|_{\partial U}=0$, $w|_{\partial U}=1$ gives
\[
 -\int_U w\,fd\vol
 =\int_U w\,\Delta_X vd\vol
 =\int_{\partial U} w\,\partial_n v\,d\sigma
 =\int_{\partial U}\partial_n v\,d\sigma\leq0, 
\]
and $w\geq0$ in $U$.

Finally, we show that $w>0$.  Write $c\doteq\div X=c^+-c^-$.  Since $w\geq0$,
the equation $\Delta_X^*w=0$ can be rewritten as
\[
 \Delta w+\langle X,\nabla w\rangle-c^-w=-c^+w\leq0 \qquad\text{in }\Omega ,
\]
so $w$ is a nonnegative supersolution of the operator
$\Delta+\langle X,\nabla\,\cdot\,\rangle-c^-$, whose zero-order coefficient
$-c^-$ is nonpositive.  If $w$ vanished at some interior point, it would
attain a nonpositive minimum in the connected set $U$, and the strong
maximum principle \cite[Theorem~3.5]{GT} (applied to $-w$) would force
$w\equiv0$ in $U$, contradicting $w=1$ on $\partial U$ and
continuity up to the boundary.  Hence $w>0$ in $U$, and also on
$\partial U$. \qed

\smallskip
Let $U_1\Subset U_2\Subset\cdots$
$\bigcup_k U_k=M$ be an exhaustion of $M$ by connected relatively compact open sets with smooth
boundary, and fix $o\in U_1$.  By Lemma~\ref{lemma:ExDirPb}, for every $k$ there is
$w_k\in C^\infty(\overline{U_k})$, $w_k>0$, with $\Delta_X^*w_k=0$ in
$U_k$. Put
\[
 u_k\doteq\frac{w_k}{w_k(o)},\qquad\text{so that } u_k>0,\quad \Delta_X^*u_k=0
 \text{ in } U _k,\quad u_k(o)=1 .
\]
We claim that for fixed $j$ and $k>j$ $u_k$ is uniformly bounded above and away from $0$ in $U_j.$ Indeed, $u_k$ is a positive
solution of the divergence-form equation with smooth coefficients $\div(\nabla u_k+u_kX)=0$ on
$U_{j+1}\Supset\overline{U_j}$.  The
interior Harnack inequality \cite[Theorem~8.20]{GT} holds on coordinate
balls contained in $U_{j+1}$, and chaining finitely many such balls
along the compact connected set $\overline{U_j}$ yields a constant
$C_j$, independent of $k$, such that
\begin{equation}\label{eq:Harnack}
 \sup_{\overline{U_j}}u_k\leq C_j\inf_{\overline{U_j}}u_k
 \leq C_j\,u_k(o)=C_j,
 \qquad
 \inf_{\overline{U_j}}u_k\geq C_j^{-1}u_k(o)=C_j^{-1},
\end{equation}
 Interior Schauder estimates
\cite[Theorem~6.2]{GT} applied on $U_{j+1}$ then give
\[
 \|u_k\|_{C^{2,\alpha}(\overline{U_j})}
 \leq C'_j\sup_{U_{j+1}}u_k\leq C'_jC_{j+1}
 \qquad(k>j+1),
\]
with $C'_j$ independent of $k$.

The Ascoli-Arzel\`a  theorem and a diagonal argument show that there is a subsequence $(u_{k_i})$ converging in $C^2(\overline{U_j})$ for every $j$.  The limit $u$ belongs to $C^2(M)$, satisfies $\Delta_X^*u=0$ on $M$ and $u(o)=1$, and by \eqref{eq:Harnack} it satisfies $u\geq C_j^{-1}>0$ on $\Omega_j$ for every
$j$; hence $u>0$ on $M$.  Finally, $u$ smooth by 
interior regularity (see \cite[Theorem~6.17]{GT}).

\end{proof}
    
It is worth pointing out that the existence of invariant measures can be obtained straightforwardly via the {\it criticality theory} for generic elliptic operators satisfying $\L 1 =0$.\smallskip

Given an elliptic operator $\L$ on $M$, define the convex cone
\[
\CC_{\L} = \{ u \in C^{2}(M): \L u =0, u>0\}.
\]
Following Y. Pinchover, \cite{Pi}, we say that $\L$ is:
\begin{itemize}
 \item [(a)]  {\it supercritical} if $\CC_{\L} = \{ u \in C^{2}(M): \L u =0, u>0\} = \emptyset$;
 \item [(b)] {\it critical} if $\CC_{\L} \not= \emptyset$ but $\L$ has no positive Green function;
 \item [(c)] {\it subcritical} if $\L$ has a positive Green function. In particular, $\CC_{\L} \not=\emptyset$.
\end{itemize}
Clearly, if $\L 1 =0$ then $\CC_{\L} \not=\emptyset$ and, therefore, $\L$ is either subcritical or critical. On the other hand, $\L$ is sub/super/critical  if and only if so is its formal adjoint $\L^{\ast}$ (see e.g. \cite[Section 4.3]{Pin}). Since $\L 1 = 0$, it follows that $\CC_{\L^\ast} \not= \emptyset$ and, consequently, there is an invariant measure for $\L$. \smallskip

Unlike general invariant measures, invariant probability measures need not exist.

\begin{example}[Nonexistence in nonnegative curvature]\label{ex_nonexistence}
 Let $(M,\metric ,\vol)$ be a complete, noncompact, Riemannian manifold with nonnegative Ricci curvature $\ric \geq 0$. Then, there exists no invariant probability measure  $d\mu = u \, d\vol$ for the Laplace-Beltrami operator. Indeed, we know from Calabi and Yau that $\vol(M) = +\infty$. Moreover,  since $\Delta^{\ast} = \Delta$, invariance of $\mu$ is equivalent to the fact that $u>0$ is harmonic. By Cheng-Yau gradient estimates, $u$ is constant and, hence, $\mu(M)=+\infty$.
\end{example}

\subsection{Uniqueness of invariant measures for the $X$-Laplacian}
From the elliptic point of view, uniqueness of the invariant measure boils down to the following Liouville type results for the adjoint operator. In view of Lemma \ref{lemma-abstractexistence} its proof is trivial. Nevertheless, it is useful to give it a formal statement.

\begin{lemma}\label{th-unique-liouville}
Let $(M,\metric, \nu )$ be a smooth metric measure space. Consider the following Liouville type properties:
\begin{itemize}
 \item [(L1)] Any  solution $0 \leq u \in C^{\infty}(M)$ of the equation $\L^{\ast,\nu} u = 0$ is constant.\smallskip
 \item [(L2)]  Any smooth solution $0 \leq u \in L^{1}(M,d\nu)$ of the equation  $\L^{\ast,\nu} u =0$ is constant.\medskip
\end{itemize}

\noindent If {\rm (L1)} holds, then any two invariant measures for $\L$ are multiple of each other, so there is essentially only one  invariant measure for $\L$ that is absolutely continuous with respect to $\nu$.\\
If {\rm (L2)} is satisfied, then there exists at most one invariant probability measure for $\L$ that is absolutely continuous with respect to $\nu$.
\end{lemma}

\begin{example}[$f$-Laplacian and $L^{1}$-Liouville]\label{ex_f-laplacian_uniqueness}
 Let $\L = \Delta_{f}$ be the weighted Laplacian on the stochastically complete weighted manifold $M_f$ and let  $\nu = \vol_f $ be its weighted measure. Then, the Liouville property $(L2)$ is satisfied. Indeed, by an observation of A. Grigor'yan, \cite[Theorem 3]{Gr}, $\Delta_{f}$ has the $L^{1}$-Liouville  property for superharmonic functions, which, in turn, implies $(L2)$. The same conclusion holds if we replace {\it stochastic completeness} with the {\it existence of an invariant probability measure $\mu$ for the semigroup $P^f_t$}. This is a consequence of Doob uniqueness result, Theorem \ref{th-Doob}. Alternatively, one can recall that, by Lemma \ref{lemma-stochcompl}, $M_f$ is stochastically complete, and then proceed as above.
\end{example}

Let $X$ be a smooth vector field on a Riemannian manifold $(M,\metric )$ and let $\mu$ be an $X$-invariant measure. It follows from Example \ref{example:X-inv} that $\mu$ is the unique $X$-invariant measure provided $M$ is parabolic for the dual operator $\Delta_X^{\ast,\mu} = \Delta_{X_\mu}$, where $X_\mu = -X + 2\nabla f$. On the other hand, as we shall see in Proposition \ref{equiv_green_Xparab}, $M$ is $X$-parabolic if and only if it is $X_\mu$-parabolic. Therefore, we obtain the following
\begin{proposition}\label{prop:uniqueness-Xinvariant}
    Let $(M,\metric)$ be a Riemannian manifold endowed with an invariant measure $\mu$ for the drifted Laplacian $\Delta_X$. If $M$ is $X$-parabolic then $\mu$ is the essentially unique invariant measure for $\Delta_X$.
\end{proposition}

\begin{example}[Nonuniqueness in negative curvature]
Let $(M,\metric,\vol)$ be a Riemannian manifold such that the cone $\H_{+}(M)$ of positive harmonic functions is non-trivial. Then  uniqueness of the invariant measure $d\mu = u  \dv$ for the Laplace-Beltrami operator fails. This follows from Lemma \ref{lemma-abstractexistence} with $\Delta^{\ast} = \Delta$ and $u \in \H_{+}(M)$. For instance, if $M$ is a Cartan-Hadamard manifold\footnote{A complete, simply connected Riemannian manifold with nonpositive sectional curvatures.} and $-B^{2} \leq \sect \leq -A^{2}$ for some $A,B>0$, then the space of bounded harmonic functions $\H_{b}(M)$, and hence  $\H_{+}(M)$, is infinite dimensional, \cite{An, Su}. For another example, consider a complete Riemannian manifold $(M,\metric)$ with at least two ends\footnote{An end of a complete manifold $M$ with respect to a smooth compact domain $\Omega$ is any of the unbounded connected components of $M\ \backslash \Omega$.} one of which is non-parabolic\footnote{An end $E$ of $M$ is parabolic if its Riemannian double $\D(E)$ is a parabolic manifold without boundary.}. Then $\H_{+}(M)$ is nontrivial (cf. \cite{LT}). This happens, e.g., if $M \hookrightarrow\rr^N$ is a complete minimal submanifold of dimension $m \geq 3$ with two ends, \cite{CSZ}, and, more generally, if the mean curvature of the immersion is suitably controlled, \cite{PRS}.
\end{example}

In the setting of probability measures, using the viewpoint of semigroups, uniqueness of invariant measures has been obtained under the validity of a {\it strong Feller property}. 
We shall denote by $B_{b}(M)$ the space of bounded Borel measurable functions.

\begin{definition}
The semigroup $P_{t}$ is said to be:
\begin{itemize}
 \item Strongly Feller if $P_{t}(B_{b}(M)) \subseteq C_{b}(M)$;
 \item Irreducible if $P_{t}1_{B_{r}}>0$ for every open ball $B_{r}$.
\end{itemize}
\end{definition}

The following Proposition provides an interesting class of operators generating irreducible and strongly Feller semigroups (see e.g. \cite [Corollary 4.7]{MPW}).

\begin{proposition}\label{prop-strongFeller}
 Let $P_{t}$ be the semigroup generated by the operator $\Delta_{X}$. Then $P_{t}$ has a positive kernel $p_{t}(x,y)$ and, in particular, it is irreducible. Moreover, $P_{t}$ is strongly Feller.
\end{proposition}

We recall now the following powerful uniqueness result, see e.g. \cite[Theorem 4.2.1]{DZ}

\begin{theorem}[Doob]\label{th-Doob}
Assume that the conservative semigroup  $P_{t}$ is strongly Feller and irreducible. Then, there exists at most one invariant probability measure for $P_{t}$.
\end{theorem}

When combined with Lemma \ref{lemma-stochcompl} and  Proposition \ref{prop-strongFeller}, Theorem \ref{th-Doob} gives the next

\begin{corollary}\label{cor-unique-XLap}
 If an invariant probability measure for the semigroup $P_t$ generated by the $X$-Laplacian $\Delta_{X}$ exists, then $P_{t}$ is stochastically complete, and its invariant probability measure is unique.
\end{corollary}

\section{Nice mollification for the $X$-Laplacian}\label{section:kernels}

The aim of this section is to extend the smooth, monotone approximation (mollification) of sub/super-harmonic functions on Riemannian manifolds, introduced in \cite{PVV}, to the general drifted operator X-Laplacian. Given $\mu$ an $X$-invariant measure, as in Example \ref{example:X-inv} we write
\[
\mu=\vol_f=e^{-f}\dv
\]
for some smooth function $f$. We recall that the formal adjoint of $\Delta_X$ with respect to $\mu$ has the form
\[
\Delta_X^{\ast,\mu}=\Delta_{X_\mu}
\]
with $X_\mu = - X + 2\nabla f$, and that the $X$-invariance of $\mu$ translates into the differential equation
\[
\div_f(X-\nabla f)=0.
\]
For any smooth precompact domain $\Omega$, let $g^{\Omega}_X(x,y)$ be the Dirichlet Green function associated with $\Delta_X$ on $(\Omega,d\mu)$, that is, the minimal positive fundamental solution of the equation (see e.g. \cite{Mi, Pin})
\[\left\lbrace
\begin{array}{cl}
\Delta^x_{X} g_X^\Omega(x,y) = - \delta^\mu_y(x) &  \ \ \text{for } x, y \in \Omega,\\[0.2cm]
g_X^\Omega(x,y) = 0  & \ \ \text{for } x, y \in \partial\Omega.
\end{array}\right.
\]

The lack of symmetry of the operator $\Delta_X$ implies that the associated Green function is also non-symmetric. Nevertheless, we have the following identity
\begin{equation*}
g_X^\Omega(x,y) = g_{X_\mu}^\Omega(y,x),
\end{equation*}
where $g_{X_\mu}^\Omega(x,y)$ denotes the Green function associated to the dual operator $\Delta_{X_\mu}$ on $(\Omega,d\mu)$. Given an exhaustion of $M$ by open sets $\Omega_k \uparrow M$, the limit
$$ g_X(x,y) = \lim_{k \rightarrow \infty} g_X^{\Omega_k}(x,y)$$
is either identically infinity, or finite outside the diagonal and satisfies
$$\Delta^x_{X} g_X(x,y) = - \delta^\mu_y(x) \quad \text{for every } \ y \in M.$$

We then define the $X$-Green operator, acting on $\varphi \in C^\infty_c(M)$, as
$$ G_X \varphi(x) = \int_M g_X(x,y)\varphi(y)d\mu(y) \in C^\infty(M).$$
In the distributional sense, we can assert that
$$\Delta^x_{X}G_X \varphi(x) = -\varphi(x) .$$
Furthermore, $g_X^{\Omega}(x,y)$ satisfies the following properties:
\begin{itemize}
\item [a)] $g_X^\Omega(x,y) = C(1 + o(1)) d(x,y)^{2-m}$ as $y \rightarrow x$, for $m\geq 3$;\smallskip
\item [b)] $g_X^\Omega(x,y) = C(1 + o(1)) \log d(x,y)^{-1}$ as $y \rightarrow x$, for $m= 2$;\smallskip
\item [c)] $\vert \nabla^{x} g^\Omega_X(x,y)\vert \leq C d(x,y)^{1-m}$ as $y \rightarrow x$.
\end{itemize}

The validity of these properties for general second-order elliptic operators with smooth coefficients is standard, although it is difficult to find all of them stated in a single reference. We limit ourselves to noting that, in our smooth setting, property $c)$ follows from classical Schauder theory, \cite{GT}. Indeed, if $u$ is a smooth solution of $\Delta_X u = 0$ on a domain $\Omega$, then $d(x,\partial \Omega)|\nabla u(x)| \leq C \sup_{\Omega} |u|$, where $C>0$ depends only on $n$, the uniform ellipticity condition in $\Omega$, the $C^{1,\a}$-norms of the metric coefficients, and the $C^{0,\a}$-norm of $X$. Now, having fixed the pole $y$ and for any $x \not= y$ close to $y$, we may apply the Schauder estimate to $g_X(x,y)$ on $\Omega = B_r(x)$ with $0< r \sim d(x,y) \ll 1$ and use the upper estimate $g_X(x,y) \leq C d(x,y)^{2-m}$ on the right hand side. Moreover, the relevant constant in the Schauder estimate depends only on $m$, the metric coefficients, and the vector field on the compact set $\bar B_1(y)$. Hence the above estimate holds uniformly on $B_{1/100}(y) \backslash \{y\}$.
\medskip

\noindent{\bf The mollification.} Let $\psi:\mathbb R \to [0,1]$ be a smooth function such that
\begin{align*}
\dot \psi \geq 0, \qquad \supp\left(\dot \psi\right)\subseteq [-1;1] \qquad \text{and} \qquad \int_{\mathbb R} \dot \psi=1\, .
\end{align*}
For ease of notation, let $g(x,y)\doteq g_X^{\Omega}(x,y)$ and $g^\ast(x,y)\doteq g_{X_\mu}^{\Omega}(x,y)$. Given $u \in L^1(\Omega,\vol_f)$, we define for $r>0$
\begin{eqnarray*}
u_r(x) &=& \int_{\Omega} K_r(x,y)u(y)   \dv_f(y),
\end{eqnarray*}
where
\[
K_r(x,y) = \dot\psi (g(x,y) - r) \vert \nabla^y g(x,y)\vert^2.
\]
Equivalently, 
\begin{eqnarray*}
u_r(x) &=& \int_{\mathbb{R}} \dot\psi (s-r)  ds \int_{g(x,y) = s} u\vert \nabla g(x,y) \vert  d\mathcal{H}^{m-1}_f \nonumber \\[0.2cm]
&=& \int_{\mathbb{R}} \dot\psi (s-r)  ds \int_{g(x,y) = s} u\langle \nabla g(x,y), - \eta \rangle   d\mathcal{H}^{m-1}_f,
\end{eqnarray*}
where we used Sard's theorem and $\eta = -\frac{\nabla g(x,y)}{\vert \nabla g(x,y)\vert}$. 

Note that, for every fixed $x, y \in \Omega$, we have $K_r(\cdot,y),\, K_r(x,\cdot) \in C^\infty_c(\Omega)$. Hence, integration by parts and the invariance of the measure yield
 \begin{eqnarray}\label{kernel_Kr}
 \int_\Omega K_r(x,y) \dv_f(y) &=&  \int_{\mathbb{R}} \dot\psi (s-r)  ds \int_{g(x,y) = s} \langle \nabla g(x,y), - \eta \rangle   d\mathcal{H}^{m-1}_f(y) \nonumber\\[0.2cm]
 &=& - \int_{\mathbb{R}} \dot\psi (s-r)  ds \int_{g(x,y) \geq s} \Delta_f g(x,y) \dv_f(y) \nonumber\\[0.2cm]
 &=& -\int_{\mathbb{R}} \dot\psi (s-r)  ds \int_{g(x,y) \geq s} (\Delta_X g(x,y) + \langle X-\nabla f, \nabla g(x,y)\rangle )\dv_f(y) \\[0.2cm]
 &=& 1 - \int_{\mathbb{R}} \dot\psi (s-r)s  ds \int_{g(x,y) = s} \langle X-\nabla f, \eta \rangle d\mathcal{H}^{m-1}_f(y) \nonumber\\[0.2cm]
 &=& 1. \nonumber
 \end{eqnarray}

Philosophically, $K_r $ plays the role of a non-symmetric convolution kernel that mollifies $L^1$-functions. In fact, if $u \in C^0_c(\Omega)$ then $u_r\in C^\infty_c(\Omega)$. Moreover, 
    \[
    \inf_\Omega u \leq \inf_\Omega u_r \leq u_r \leq \sup_\Omega u_r \leq \sup_\Omega u. 
    \]
    In particular, $0 \leq u \leq 1$ implies the same bounds for $u_r$. \medskip

\noindent{\bf Representation formula.} We prove the following Green representation formula: for every $\varphi \in C^2(\Omega_1)\cap C^1(\overline{\Omega_1})$, where $\Omega_1 \subset \Omega$, 
\begin{multline}\label{green_repr}
\int_{\Omega_1} g(x,y) \Delta_{X} \varphi\dv_f + \varphi(x) =\\ \int_{\partial \Omega_1} \langle g(x,y) \nabla \varphi - \varphi \nabla g(x,y) - \varphi g(x,y) (X-\nabla f), \eta\rangle d\mathcal{H}^{m-1}_f.
\end{multline}

To show the validity of \eqref{green_repr} we first observe that, for $\varphi, \varphi_1\in C^2(\Omega_1)\cap C^1(\overline{\Omega_1})$ we have
$$\int_{\Omega_1} \varphi_1 \Delta_X \varphi \dv_f = \int_{\Omega_1} \varphi \Delta_{X_\mu} \varphi_1 \dv_f + \int_{\partial \Omega_1} \langle \varphi_1 \nabla \varphi - \varphi \nabla \varphi_1 - \varphi\varphi_1 (X-\nabla f), \eta\rangle d\mathcal{H}^{m-1}_f. $$
Applying the above equality for $\varphi \in C^2(\Omega_1)\cup C^1(\overline{\Omega_1})$ and $g(x,y)$ on $\Omega_1 \backslash B_{\varepsilon}(x)$, we get
\begin{multline*}
\int_{\Omega_1 \backslash B_{\varepsilon}(x)} g(x,y) \Delta_{X} \varphi \dv_f = 
\int_{\Omega_1 \backslash B_{\varepsilon}(x)} \varphi \Delta^y_{X_\mu} g(x,y) \dv_f +\\[0.2cm]
 \int_{\partial \Omega_1 \cup \partial B_{\varepsilon}(x)} \langle  g(x,y)  \nabla \varphi -\varphi \nabla g(x,y) - \varphi g(x,y) (X-\nabla f), \eta\rangle d\mathcal{H}^{m-1}_f=\\[0.2cm]
 \int_{\Omega_1 \backslash B_{\varepsilon}(x)} \varphi \Delta^y_{X_\mu} g^\ast(y,x) \dv_f +\\[0.2cm]
 \int_{\partial \Omega_1 \cup \partial B_{\varepsilon}(x)} \langle  g(x,y)  \nabla \varphi -\varphi \nabla g(x,y) - \varphi g(x,y) (X-\nabla f), \eta\rangle d\mathcal{H}^{m-1}_f.
\end{multline*}    
It follows from the asymptotic behavior of $g(x,y)$ and $\nabla g(x,y)$, as $\varepsilon\rightarrow 0$, that 
\begin{equation*}
\int_{\Omega_1} g(x,y) \Delta_{X} \varphi \dv_f  +\varphi(x) = \int_{\partial \Omega_1} \langle g(x,y) \nabla \varphi - \varphi \nabla g(x,y) - \varphi g(x,y) (X-\nabla f), \eta\rangle d\mathcal{H}^{m-1}_f.
\end{equation*}
\medskip

\noindent{\bf Monotonicity:} Assume that $\Delta_{X} u \geq 0$ in the sense of distributions in $\Omega$. 
Then, for every $\Omega'\Subset \Omega$, $u_r(x)$ is eventually monotonically decreasing in $r$ for every $x\in \Omega '$. Indeed, 
\begin{equation}\label{derivative_ur}
\frac{\partial}{\partial r} u_r(x) = \int_{\mathbb{R}} \ddot\psi (s-r)  ds \int_{g(x,y) = s} u\langle \nabla g(x,y), \eta \rangle   d\mathcal{H}^{m-1}_f.
\end{equation}

Without loss of generality, we first assume $u\in C^2(\Omega)$. Using \eqref{green_repr} we can compute
\begin{equation}\label{comp_int_g}
\begin{split}
\int_{g(x,y) = s} u \langle \nabla g(x,y),\eta\rangle \, d\mathcal{H}^{m-1}_f + u(x)
= - \int_{g(x,y) > s} g(x,y) \Delta_{X} u  \, \dv_f
\\
\qquad + \int_{g(x,y) = s} g(x,y) 
\langle \nabla u - u(X-\nabla f), \eta \rangle \, d\mathcal{H}^{m-1}_f
\\
= - \int_{g(x,y) > s} g(x,y) \Delta_{X} u \, \dv_f
+ s\int_{g(x,y) = s}
\langle \nabla u - u(X-\nabla f), \eta \rangle \, d\mathcal{H}^{m-1}_f
\\
= - \int_{g(x,y) > s} (g(x,y) - s)\Delta_{X} u \, \dv_f.
\end{split}
\end{equation}

Hence, by co-area
\begin{equation*}
\int_{g(x,y) = s} u \langle \nabla  g(x,y), \eta\rangle d\mathcal{H}^{m-1}_f = - u(x) - \int_s^\infty (t-s) dt \int_{g(x,y) = t} \frac{\Delta_{X} u}{\vert \nabla g(x,y) \vert} d\mathcal{H}^{m-1}_f.
\end{equation*}
Inserting this  into \eqref{derivative_ur}, using the fact that $\dot \psi$ is supported in $[-1,1]$, and integrating by parts give
\begin{eqnarray*}
    \frac{\partial}{\partial r} u_r(x) &=& -\int_{\mathbb{R}} \ddot\psi (s-r)  ds\left( u(x) + \int_s^\infty (t-s) dt \int_{g(x,y) = t} \frac{\Delta_{X} u}{\vert \nabla g(x,y) \vert} d\mathcal{H}^{m-1}_f\right)\\
    &=& -\int_{\mathbb{R}} \ddot\psi (s-r)  ds\left( \int_s^\infty (t-s) dt \int_{g(x,y) = t} \frac{\Delta_{X} u}{\vert \nabla g(x,y) \vert} d\mathcal{H}^{m-1}_f\right)\\
    &=& -\int_\mathbb{R} \dot\psi(s-r) \int_s^\infty dt \int_{g(x,y) = t} \frac{\Delta_{X} u}{\vert \nabla g(x,y) \vert} d\mathcal{H}^{m-1}_f \\
    &=& -\int_\mathbb{R} \dot\psi(s-r) 
    \int_{g(x,y) > s} \Delta_{X} u \dv_f \\
    &=&  -\int_\Omega \Delta_{X} u \int_{-\infty}^{g(x,y)} \dot \psi(s-r) ds \dv_f \\
    &=& -\int_\Omega \psi(g(x,y)-r)\Delta_{X} u \dv_f. 
\end{eqnarray*}

Now, let $\Omega' \Subset \Omega'' \Subset \Omega$. By the definition of $\psi$, there exists $r_0$ such that, for all $r\geq r_0$ and for all $x\in \Omega'$, $\psi(g(x,y) - r)$ is compactly supported in $\Omega''$. Thus  $0\leq \psi(g(x,y) - r) \in C^\infty_c(\Omega'')$ and we conclude that 
\begin{equation*}
    \frac{\partial}{\partial r} u_r(x) = - \int_\Omega u(y) \Delta_{X_\mu}\psi(g(x,y) - r) \dv_f(y).
\end{equation*}
Note that the above identity extends to $u\in L^1_{\mathrm{loc}}(\Omega)$. Since $\Delta_{X} u \geq 0$ in the sense of distributions, we deduce that
$$  \frac{\partial}{\partial r} u_r(x) \leq 0.$$

\medskip

\noindent{\bf Subharmonicity:} 
We now prove that if $\Delta_{X}u \geq 0$ in $\Omega$ then $\Delta_{X}u_r \geq 0$ in $\Omega' \Subset \Omega$. We first rewrite $u_r$ using the identity \eqref{comp_int_g} as follows 
\begin{eqnarray*}
u_r(x) &=& \int_{\mathbb{R}} \dot\psi (s-r)  ds \int_{g(x,y) = s} u\langle \nabla g(x,y), - \eta \rangle   d\mathcal{H}^{m-1}_f \\[0.2cm]
&=& u(x) +\int_{\mathbb{R}} \dot\psi (s-r)  ds \int_{g(x,y) > s} (g(x,y) - s)\Delta_{X} u \dv_f.
\end{eqnarray*}
Thus,
\begin{eqnarray*}
u_r(x) - u(x) &=& \int_{\mathbb{R}} \dot\psi (s-r)  ds \int_{g(x,y) > s} (g(x,y) - s)\Delta_{X} u \dv_f\\[0.2cm]
&=& \int_{\Omega} \Delta_{X} u \dv_f \int_{\mathbb{R}} (g(x,y) - s) \dot\psi (s-r) \chi_{g(x,y) > s}     ds.
\end{eqnarray*}

Setting $\hat{\psi}(t) = \int_{-\infty}^t \psi(s) ds$ and using identity (2.17) from \cite{PVV}, we get
\begin{eqnarray*}
u_r(x) - u(x) = \int_{\Omega} \hat{\psi}(g(x,y) - r) \Delta_{X} u  \dv_f.
\end{eqnarray*}
Therefore, 
\begin{eqnarray*}
\Delta_{X} u_r(x) - \Delta_{X} u(x) &=& \int_{\Omega} \Delta^x_{X} \hat{\psi}(g(x,y)(y) - r) \Delta_{X} u(y)  \dv_f(y) \\[0.2cm]
&=& \int_{\Omega} [-\psi(g(x,y) - r)\delta^\mu_y(x) + \dot\psi(g(x,y) - r)\vert \nabla^x g(x,y)\vert^2 ] \Delta_{X} u  \dv_f ,
\end{eqnarray*}
where $\nabla^x$ and $\Delta_X^x$ denote the operators acting on the variable $x$. Now, for every $\varphi\in C^\infty_c(\Omega)$ we can compute
\begin{multline*}
\int_\Omega \varphi(x)\Delta_{X} u_r(x)\dv_f(x) - \int_\Omega \varphi(x)\Delta_{X} u(x) \dv_f(x) = \\[0.2cm]
\int_\Omega \varphi(x) \dv_f(x) \int_{\Omega} [-\psi(g(x,y) - r)\delta^\mu_y(x) + \dot\psi(g(x,y) - r)\vert \nabla^x g(x,y)\vert^2 ] \Delta_{X} u(y)  \dv_f(y)=\\[0.2cm]
\int_\Omega \Delta_{X} u(y)  \dv_f(y) \int_{\Omega} [-\varphi(x)\psi(g(x,y) - r)\delta^\mu_y(x) + \varphi(x)\dot\psi(g(x,y) - r)\vert \nabla^x g(x,y)\vert^2 ]  \dv_f(x)=\\[0.2cm]
-\int_\Omega \varphi(y)\Delta_{X} u(y)  \dv_f(y) +  \int_{\Omega} \varphi(x) \dv_f(x) \int_\Omega \dot\psi(g(x,y) - r)\vert \nabla^x g(x,y)\vert^2 \Delta_X u(y)  \dv_f(y).
\end{multline*}
Thus,
$$\Delta_{X} u_r(x) = \int_{\Omega} \dot{\psi}(g(x,y) - r)\vert \nabla^x g(x,y)\vert^2 \Delta_{X} u\dv_f$$
in the sense of distributions.
\medskip

\noindent{\bf $L^1$-convergence:}
We will establish that $u_r$ converges to $u\in L^1_{\rm loc}$ following the same steps from \cite{PVV}. We first recall that 
$$ u_r(x) - u(x) = \int_\Omega \hat{\psi}(g(x,y) - r) \Delta_{X} u(y) \dv_f(y).$$
Integrating in a smooth domain $\Omega' \Subset \Omega'' \Subset \Omega$, we obtain
\begin{eqnarray*}
    \int_{\Omega'} \vert u_r(x) - u(x)\vert \dv_f(x) &=& \int_{\Omega'}  [u_r(x) - u(x)] \dv_f(x) \\[0.2cm]
    &=& \int_{\Omega} \Delta_{X} u (y) h(y,r) \dv_f(y),
\end{eqnarray*}
where $h(y,r) = \int_{\Omega'} \hat{\psi}(g(x,y) - r) \dv_f(x)$. The next step is to prove that $h(y,r)$ is a Lipschitz function in $y \in \Omega$. Since 
$$ \nabla h = \int_{\Omega'} \psi(g(x,y)-r)\nabla g(x,y) \dv_f(x),$$
the claim follows by gradient estimates for $g(x,y)$. Once we have proved that $h(y,r)$ is Lipschitz, we can proceed to show that $u_r$ is a $L^1$ Cauchy sequence in $r$. Indeed, for $R>r>r_0$ it holds
\begin{eqnarray*}
\int_{\Omega'} \vert u_r(x) - u_R(x)\vert \dv_f(x) &=& \int_{\Omega'} (u_r(x) - u_R(x)) \dv_f(x) \\[0.2cm]
&=& \int_\Omega \Delta_{X} u(y) H(y,r,R) \dv_f(y),
\end{eqnarray*}
where 
$$H(y,r,R) = \int_{\Omega'}[\hat{\psi}(g(x,y) - r) - \hat{\psi}(g(x,y) - R)] \dv_f(x).$$
As in \cite{PVV}, it follows from the uniform boundedness of $H(y,r,R)$ that $u_r$ is a $L^1$ Cauchy sequence.\medskip

\noindent{\bf Local Pointwise Convergence:}
For $x\in \Omega'$, we can show that $u_r(x) \rightarrow u(x)$ as $r \rightarrow \infty$. The argument follows the same lines from \cite{PVV}. We only need the Cheng-Yau gradient estimate for the drifted Laplacian, see e.g. \cite{Lin}, and the fact that 
$$C^{-1} d(x,y)^{2-m} \leq g(x,y) \leq C d(x,y)^{2-m} \quad \text{for } \ x \in \Omega', y \in \Omega''.$$

\smallskip

\section{Parabolicity for the drifted Laplacian}\label{section:X-parab}

As we have briefly recalled in the Introduction, there are several ways to characterize the parabolicity of a weighted Riemannian manifold with respect to its weighted Laplace-Beltrami operator. In strong contrast, to the best of our knowledge, there are only a few results in the literature regarding the parabolicity of a manifold with respect to the drifted Laplace-Beltrami operator $\Delta_X$, where $X$ is a smooth vector field defined on $M$.

\begin{definition}
A manifold $M$ is said to be parabolic with respect to the drifted Laplace-Beltrami operator $\Delta_X$, $X$-parabolic for short, if it satisfies the following $X$-Liouville property:
$$w\in C^0(M)\cap W^{1,2}_{\text{loc}}(M), \quad w\geq 0, \quad \Delta_X w \leq 0 \quad \Longrightarrow \quad w \text{ is constant}.$$
\end{definition}

Clearly if $X = \nabla f$, for some smooth function $f$, we recover the notion of $f$-parabolicity. In \cite{CMMR}, as a nontrivial extension of the classical Khasminskii criterion valid for the (possibly weighted) Laplacian, the authors proved the following result.

\begin{lemma}\label{lemma-parabolic}
Let $M$ be a manifold and $X$ be a given smooth vector field. Assume that there exists a function $\g \in \text{\rm Lip}_{loc}(M)$ such that:
\begin{itemize}
\item [(j)] $\g(x) \to \infty$ as $x \to \infty$;
\item [(jj)] $\Delta_{X} \g \leq 0$ outside a compact set.
\end{itemize}
Then $M$ is $X$-parabolic.
\end{lemma}

As a consequence of \cite[Proposition 4.2]{MP} the $X$-parabolicity can be also characterized in terms of the Ahlfors property: for every open subset $\Omega \subset M$, and every $w \in C(\overline\Omega)$ bounded from above such that $\Delta_X w \geq 0$, it holds
$$ \sup_{\overline{\Omega}} w = \sup_{\partial \Omega} w.$$
Noting that $\Delta_X$ satisfies the strong maximum principle and the comparison principle, we can conclude from \cite[Theorem 4.4]{MP} (see also \cite[Theorem 3.1]{MP1}) that the above Khasminskii property is indeed equivalent to the Ahlfors property, and consequently, equivalent to the $X$-parabolicity of $M$. In the following, we describe the characterization of the $X$-parabolicity in terms of the $X$-Green function.

\begin{proposition}\label{equiv_green_Xparab}
Let $(M,\metric)$ be a Riemannian manifold endowed with an $X$-invariant measure $d\mu =e^{-f}d\vol$. The following are equivalent:
\begin{itemize}
\item[1.] $M$ is parabolic for $\Delta_X$;
\item[2.] There is no positive $X$-Green function, i.e., $g_X(x,y) \equiv \infty$;
\item[3.] $M$ is parabolic for $\Delta^{\ast,\mu}_X = \Delta_{X_\mu}$, where $X_\mu = -X + 2\nabla f$.
\end{itemize}
\end{proposition}
\begin{proof}
{\bf $1. \Rightarrow 2.$:} Assume by contradiction that $ g_X(x,y) < \infty$. Let $\varphi \in C^\infty_{c,+}(M)$ and note that the function
$$u(x) =  G_X\varphi(x)=  \int_M g_X(x,y)\varphi(y)\dv_f(y)$$
is a non-constant function satisfying $u \geq 0$, $\Delta_{X} u = -\varphi \leq 0$. Thus, $M$ cannot be $X$-parabolic.\\

\noindent{\bf $2. \Rightarrow 1.$:}
Let $u \in W^{1,2}_{loc}(M) \cap C^0(M)$ be an $X$-subharmonic function satisfying $\sup_M u = u^\ast <+\infty$. We have to show that $u$ is constant. Following \cite{SNWC}, we proceed in two steps.\smallskip

\noindent {\bf Step 1.} We first prove that if $w$ is a non-positive $X$-subharmonic function, then $w$ is, in fact, $X$-harmonic. To make the argument more transparent, let us assume for the moment that $w \in C^2$. Fix a smooth precompact exhaustion $\Omega_n$ of $M$ and let $g_n(x,y)\doteq g^{\Omega_n}_X(x,y)$ be the Dirichlet Green function of $\Delta_X$ on $\Omega_n$. Let $w_n$ be the solution of the problem
\[
\begin{cases}
    \Delta_X w_n = \Delta_X w &\Omega_n,\\
    w_n = 0 & \partial \Omega_n,
\end{cases}
\]
which is given explicitly by
\[
w_n(x) = \int_{\Omega_n} g_n(x,y)(-\Delta_X w)(y) \dv_f(y).
\]
By the maximum principle
\begin{equation}\label{eq:subharm1}
w \leq w_n \quad \text{in }\Omega_n.
\end{equation}
Now, assume by contradiction that 
\[
\Delta_X w \geq \delta >0
\]
on some compact set $K \subset \Omega_n$, $n \gg 1$, with $\vol_f(K)>0$. Then
\[
w_n(x) \leq -\delta \int_K g_n(x,y) \dv_f(y)
\]
and since $g_n(x, \cdot)$  converges to $g_X(x,y) \equiv + \infty$ on compact sets we deduce
\[
w_n(x) \to -\infty \quad \text{as }n\to +\infty
\]
at any $x \not\in K$. This contradicts \eqref{eq:subharm1} and proves that $u$ must be $X$-harmonic. With minor changes, the argument can be adapted to the needed regularity assumption on $w$. Indeed, the distribution $\lambda=\Delta_Xw$ is a nonnegative Radon measure. On the smooth exhaustion $\Omega_n$, we may subtract from $w$ the $X$-harmonic extension of its boundary data to obtain a function $w_n$ satisfying $w\leq w_n\leq0$. Since $w$ has a $X$-harmonic majorant, we have 
\[
w_n(x)=-\int_{\Omega_n}g_{n}(x,y)\,d\lambda(y).
\]
If $\lambda\ne0$, we can choose a compact set $K$ with $\lambda(K)>0$ and $x\notin K$ such that, by monotone convergence, $w_n(x)\to-\infty$. Again, this contradicts \eqref{eq:subharm1}. \smallskip

\noindent {\bf Step 2.} Starting from the given $X$-subharmonic function $u$ we now construct a new function $w = f(u)$ where
\[
f(t)= \frac{1}{u^\ast-t + 1}, \quad t \in (-\infty,u^\ast).
\]
Clearly, $w \leq 1$. Moreover, since $f$ is strictly increasing and strictly convex, $w$ is still $X$-subharmonic. In fact
\[
\Delta_X f(u) = f'(u) \Delta_X u + f''(u) |\nabla u|^2.
\]
According to {\bf Step 1}, $w$ must be $X$-harmonic and this forces $u$ to be constant. The proof is completed.\\

\noindent{\bf $2. \Leftrightarrow 3.$:} We just note that $g_X(x,y) = g_{X_\mu}(y,x)$ and $1. \Leftrightarrow 2.$ holds with $X_\mu$ in place of $X$.
\end{proof}

For our purposes, it is convenient to introduce a non-variational capacity for the drifted Laplacian in the spirit of the capacity defined in \cite{Ba} (see also \cite{DKK}).

\begin{definition}\label{def:X-capacity}
Let $(M,\metric,\mu)$ be a metric measure space, and let $K \subset U\subset M$ two subsets such that $K$ is compact and $U$ is open. For a smooth vector field $Y$, we define the $Y$-capacity of the condenser $(K,U)$ by
$$\CAP_{Y,\mu}(K,U) = \inf\left\lbrace \int_U \vert \Delta_{Y} \varphi \vert d\mu \colon \varphi \in C^\infty_{c}(U), \  \varphi \geq 1 \ \text{in some neighborhood of }  K\right\rbrace.$$
when $U=M$ we simply write $\CAP_{Y,\mu}(K) = \CAP_{Y,\mu}(K,M)$.
\end{definition}

We also consider the following extended Green formula for bounded functions in the same vein of \cite{Gr1}.

\begin{definition}\label{def:X-GreenFormula}
The metric measure space $(M,\metric, \mu)$ is said to satisfy the $Y$-Green formula if
$$ \int_M \Delta_{Y} w \, d\mu = 0$$
for every function $w \in L^\infty(M)$ such that $\Delta_Y w \in L^1_{\text{loc}}(M,\mu)$ and $\left(\Delta_{Y} w\right)_+ \in L^1(M,\mu)$.
\end{definition}

The main result of this section is the following characterization of $X$-parabolicity.

\begin{theorem}\label{thm_equiv_Xparab}
Let $(M,\metric)$ be a Riemannian manifold endowed with an $X$-invariant measure $d\mu =e^{-f}d\vol$. The following are equivalent characterizations of the $X$-parabolicity:
\begin{itemize}
\item[1.] There is no positive $X$-Green function, i.e., $g_X(x,y) \equiv \infty$;
\item[2.] $\displaystyle G_X\psi(x) \equiv \infty$, for every $x \in M$ and every $\psi \in C^\infty_{c,+}(M)$;
\item[3.] $\CAP_{X,\mu}(\bar K)=0$ for every precompact set $K$;
\item[4.] $(M,\metric,\mu)$ satisfies the $X$-Green formula.
\end{itemize}
\end{theorem}
\begin{proof}
\noindent{\bf $1. \Rightarrow 2.$:} Obvious.\\

\noindent{\bf $2. \Rightarrow 3.$:}
Without loss of generality we can take $K$ be a precompact set with smooth boundary. Let 
 $\psi \in C^\infty_{c,+}(M)$ be such that $\supp \psi$ is contained in the interior of $K$ and 
$$ \int_M \psi(x) \dv_f(x) = 1.$$
 Let $U_n \uparrow M$ be an increasing exhaustion of $M$ by precompact sets such that $\bar K \subset U_1$. 
 Set $b_n = \min_{x \in \partial K} G_X^{U_n}\psi(x)$, so that, by assumption, $b_n \to \infty$ as $n\to \infty$,  and define $F_{b_n} = \{x \in U_n \colon  G_X^{U_n}\psi(x) \geq b_n\}$. Since $\Delta_{X} G_X^{U_n}\psi(x) = -\psi(x) \leq 0$, the minimum principle gives
$$ \min_{x \in K}  G_X^{U_n}\psi(x) = \min_{x \in \partial K} G_X^{U_n}\psi(x) = b_n$$
and then, $K \subset F_{b_n}$ for every $n\geq 1$. Since the $X$-capacity is monotone, we have
$$ \CAP_{X,\mu}(\bar K,U_{n+1}) \leq \CAP_{X,\mu}(F_{b_n},U_{n+1}).$$

Now, we consider the function $\phi = G_X^{U_n}\psi \in C^\infty_0(U_n)$, and extend it by $0$ in the subset $U_{n+1}\backslash U_n$. We then approximate $\phi$ by a function $\phi_r \in C^\infty_c(U_{n+1})$ constructed in Section~\ref{section:kernels} using $g(x,y)\doteq g_X^{U_{n+1}}(x,y)$, which satisfies
$$\phi_r \rightarrow \phi \ \ \text{locally uniformly on compact subsets of }   U_n,$$
and 
$$ \Delta_X \phi_r(x) =  \int_{U_{n+1}} K_r(y,x) \Delta_{X} \phi(y) \dv_f(y),$$
for every $y \in U_{n+1}$.
Moreover, we can also see that 
$$\Delta_{X} \phi = - \psi + \nu,$$ 
where $\nu$ is a Radon measure concentrated on $\partial U_n$, whose total mass is given by
$$\nu(M) = \int_{U_n}\psi \dv_f. $$
Identifying, with a slight abuse of notation, the total variation $|\Delta_X \phi|$ with an $L^1$-function, we have
$$\int_{U_{n+1}} \vert \Delta_{X} \phi\vert \dv_f = 2 \int_{U_n} \psi \dv_f = 2$$
which, together with \eqref{kernel_Kr}, yields
\begin{eqnarray*}\int_{U_{n+1}} \vert \Delta_{X} \phi_r\vert \dv_f &\leq &   \int_{U_{n+1}}\vert \Delta_{X} \phi(y)\vert \dv_f(y) \int_{U_{n+1}}K_r(y,x) \dv_f(x) = 2.
\end{eqnarray*}

The function $u(x) = 2b_n^{-1} \phi_r(x)$ defined on $U_{n+1}$ is a test function for $\CAP_{X,\mu}(F_{b_n},U_{n+1})$ and satisfies
\begin{eqnarray*}
\CAP_{X,\mu}(F_{b_n},U_{n+1}) &\leq & \int_{U_{n+1}} \vert \Delta_{X} u(x)\vert \dv_f(x)\\[0.2cm]
&=& \frac{2}{b_n }\int_{U_{n+1}}  \vert \Delta_X \phi_r(x)\vert \dv_f(x)\\[0.2cm]
&\leq & \frac{4}{b_n }.
\end{eqnarray*}
We conclude by letting $n \rightarrow \infty$.\\

 \noindent{\bf $3. \Rightarrow 4.$:}
The proof follows the lines of \cite[Theorem 1.1]{GM}. Assume that $\CAP_{X,\mu}(\bar K) = 0$ for every precompact subset $K\subset M$, and let $w\in L^\infty(M)$ 
be a non-constant function such that {$\Delta_{X} w\in L^1_{loc}$} and $(\Delta_{X} w)_+ \in L^1(M,\mu)$. 
By translating and scaling, we may assume that $0<\inf_{M} w <\sup_{M} w < 1$. 

Given a precompact set $K$ and a precompact open set $U$ with $K \subset U$, let $\varphi
\in C^\infty_c(U)$ be such that $\varphi \geq 1$ in a neighborhood of $K$ and 
\[
\int_U|\Delta_{X} \varphi|d\mu \leq 2 \CAP_{X,\mu}(K, U).
\]
Assume for the moment that $w$ is smooth on $U$. We set
 $v = \varphi - w$ and note that $v$ attains a positive minimum on $\partial K$ and $v < 0$ on $\partial U$. 
By Sard's theorem, there is a regular value $\varepsilon \in (0,1)$ for $v$ such that $K \subset \Omega = \{v > \varepsilon\} \subset U$ and
$$\int_{\partial \Omega} \frac{\partial v}{\partial \eta} d\mathcal{H}^{m-1}_f  \leq 0,$$
where $\eta$ is the outward normal vector field along $\partial\Omega$. 
Thus,
\begin{equation}\label{ineq_funct}
\int_{\partial \Omega} \frac{\partial w}{\partial \eta}d\mathcal{H}^{m-1}_f \geq \int_{\partial \Omega} \frac{\partial \varphi}{\partial \eta}  d\mathcal{H}^{m-1}_f.
\end{equation}

Now, since $\dv_f$ is an invariant measure, we can use that $\div_f(X -\nabla f) = 0$ to obtain
\begin{eqnarray*}
    -\int_\Omega \Delta_{X} w \dv_f &=& -\int_\Omega \Delta_f w d\vol_f + \int_\Omega \langle X- \nabla f, \nabla w\rangle \dv_f \\[0.2cm]
    &=& -\int_{\partial \Omega} \frac{\partial w}{\partial \eta}  d\mathcal{H}^{m-1}_f + \int_{\partial \Omega} w \langle X- \nabla f, \eta \rangle d\mathcal{H}^{m-1}_f .
\end{eqnarray*}
Using that $w = \varphi - \varepsilon$ on $\partial \Omega$ and the invariance of the measure, inequality \eqref{ineq_funct} yields
\begin{eqnarray*}
   - \int_\Omega \Delta_{X} w \dv_f &\leq &  -\int_{\partial \Omega} \frac{\partial w}{\partial \eta}  d\mathcal{H}^{m-1}_f + \int_{\partial \Omega} \varphi \langle  X-\nabla f , \eta \rangle d\mathcal{H}^{m-1}_f \\[0.2cm]
    &=& -\int_\Omega \Delta_f \varphi d\vol_f + \int_\Omega \langle X- \nabla f, \nabla \varphi \rangle \dv_f\\[0.2cm]
    &=& -\int_\Omega \Delta_{X} \varphi \dv_f  
     \\[0.2cm]
    &\leq&  \int_U \vert\Delta_{X} \varphi\vert \dv_f\leq 
 2 \CAP_{X,\mu}(K,U).
\end{eqnarray*}

When $w$ is not smooth the same inequality can be obtained by applying Friedrichs mollifiers (cf. \cite[Proposition 5.14]{BMS}). Indeed, since $w \in L^1_{loc}$ with $\Delta_X w \in L^1_{loc}$, Kato inequality yields the existence of a smooth sequence $w_k$ such that $w_k \to w$ in $L^1_{loc}$ and $\Delta_X w_k \to \Delta_X w$ in $L^1_{loc}$. Thus, we may apply the argument to $w_k$ on $U$ and then take the limit.

Now, let $U_n$ be an increasing exhaustion of $M$ by precompact open sets such that $\overline U_n\subset U_{n+1}$ for every $n$. Since the capacity of every compact set is zero, for every $n$ there exists $m(n)>n$ such that 
\[
\CAP_{X,\mu}(\overline{U}_n, U_{m(n)})<\frac 1n.
\]
Applying the above argument to the condenser $(\overline{U}_n, U_{m(n)})$, for every $n$
we find an open subset with smooth boundary $\Omega_n$ with  $\overline{U}_n\subset\Omega_n\subset  U_{m(n)} $ such that 
\[
-\int_{\Omega_n}\Delta_{X} w \dv_f = -\int_{\Omega_n} (\Delta_{X} w)_+ \dv_f +\int_{\Omega_n} (\Delta_{X} w)_- \dv_f
\leq 
2\CAP_{X,\mu}(\overline{U}_n, U_{m(n)})<\frac 2 n.
\]
Since $ (\Delta_{X} w)_+\in L^1(M, \mu)$, letting $n\to \infty$ and applying monotone convergence, we deduce that  $(\Delta_{X} w)_-\in L^1(M, \mu)$, so that $\Delta_{X} w \in L^1(M, \dv_f)$ and  
\[
\int_{M}\Delta_{X} w \dv_f\geq 0.
\]
The argument applied to $1-w$ then shows that 
\[
\int_{M}\Delta_{X} w \dv_f\leq 0, 
\]
and we conclude that 
\[
\int_{M}\Delta_X w \dv_f = 0.
\]
\noindent{\bf $4. \Rightarrow 1.$:} To prove this implication, we follow the same argument used for $1. \Rightarrow 2.$. Let $\varphi \in C_{c,+}^\infty(M)$, and let $g_X(x,y)$ denote the finite positive Green function. Then, $v = G_X\varphi$ belongs to $L^\infty$ and satisfies $\Delta_{X} v = - \varphi \leq 0$. Thus, 
$$\int_M \Delta_X v \dv_f = - \int_M \varphi \dv_f < 0,$$
which contradicts the $X$-Green formula.
\end{proof}

As anticipated in Remark~\ref{rmk_brezis-marcus-ponce-equiv}, Theorem~\ref{thm_equiv_Xparab} extends to the $X$-Laplacian an equivalence of the vanishing of different capacities first established by Brezis, Marcus and Ponce \cite{BMP} in the Euclidean setting.

\begin{corollary}\label{cor:equivalence}
Let $M_f$ be a weighted manifold. The following are equivalent:
\begin{enumerate}
    \item[i)] $M_f$ is $f$-parabolic.\smallskip 
    \item[ii)] The $f$-capacity of all precompact sets vanishes. Namely:
    $$ \CAP_f(K) \doteq \inf\left\lbrace \int_M \vert \nabla \varphi \vert^2 d\vol_f \colon \varphi \in C^\infty_{c}(M), \varphi \geq 1 \  \text{in a neighborhood of }  K\right\rbrace = 0$$
    for every precompact set $K\subset M$. \smallskip
\item[iii)] The $X=\nabla f$-capacity of all precompact sets vanishes. Namely:
$$ \CAP_{\Delta_f}(K) \doteq \inf\left\lbrace \int_M \vert \Delta_f \varphi \vert d\vol_f \colon \varphi \in C^\infty_{c}(M), \varphi \geq 1 \  \text{in a neighborhood of }  K\right\rbrace = 0,$$
for every precompact set $K\subset M$.    
\end{enumerate}    
\end{corollary}

We end this section by showing a hierarchy between the $X$-capacity and the $f$-capacity for every $X$-invariant measure of the form $\mu = e^{-f}\dv$, for some smooth function $f$.

\begin{proposition}\label{controllo_fCap_XCap}
Let $(M,\metric)$ be a complete Riemannian manifold endowed with an $X$-invariant measure $\mu= e^{-f}\vol$. Then, for every condenser $(K,U)$, with $K\Subset U$ and $U$ open, one has
\[
\CAP_f(K,U) \leq \frac{1}{2}\CAP_X(K,U).
\]
\end{proposition}
\begin{proof}
Let $\varphi \in C^\infty_c(U)$ be a function satisfying $\varphi\geq 1$ in some neighborhood of $K$, and consider $\tilde \varphi = \min\{1,\varphi_+\} \in C_c(U)\cap H^1_0(U)$ such that $0\leq \tilde \varphi \leq 1$. We recall that, by the variational characterization of the $f$-capacity and the standard truncation argument, we may compute $\CAP_f(K,U)$ using functions in $H^1_0(U)$ that are greater than or equal to $1$ in a neighborhood of $K$. Since $\widetilde\varphi\equiv 1$ in a neighborhood of $K$, it is an admissible competitor for $\CAP_f(K,U)$, and thus
\begin{eqnarray*}
\text{\rm Cap}_f(K,U) &\leq & \int_U \vert \nabla \tilde \varphi\vert^2 \dv_f \\[0.2cm]
&=& \int_U \langle \nabla \varphi, \nabla \tilde \varphi\rangle \dv_f\\[0.2cm]
&=& -\int_U  \tilde \varphi\Delta_f \varphi \dv_f.
\end{eqnarray*}
Using the definition of $\Delta_X$, and the fact that $0\leq \tilde \varphi \leq 1$, we obtain
\begin{eqnarray*}
\text{\rm Cap}_f(K,U) &\leq &  -\int_{U}  \tilde \varphi \Delta_X \varphi \dv_f -  \int_U  \tilde \varphi \langle X-\nabla f, \nabla \varphi \rangle \dv_f \\[0.2cm]
&= & - \int_{U}  \tilde \varphi \Delta_X \varphi \dv_f +  \int_{U}  \varphi \langle X-\nabla f, \nabla \tilde \varphi \rangle \dv_f \\[0.2cm]
&= & - \int_{U}  \tilde \varphi \Delta_X \varphi \dv_f +  \int_{U}  \langle X-\nabla f, \frac{1}{2}\nabla \tilde \varphi^2 \rangle \dv_f \\[0.2cm]
&= &  \int_{U}  (\frac{1}{2} - \tilde\varphi) \Delta_X \varphi \dv_f - \frac{1}{2} \int_{U} \tilde \varphi^2 \div_f(X-\nabla f) \dv_f \\[0.2cm]
&\leq & \frac{1}{2} \int_{U}  \vert \Delta_X \varphi \vert \dv_f, 
\end{eqnarray*}
where we have used the $X$-invariance of the measure. Therefore, taking the infimum over all $\varphi\in C^\infty_c(U)$ such that $\varphi\geq 1$ on $K$, we conclude that
$$\text{\rm Cap}_f(K,U) \leq \frac{1}{2}\text{\rm Cap}_X(K,U).$$ 
\end{proof}

\begin{remark}
We point out that, under the additional $L^2$-defect assumption on the $X$-invariant measure, see Definition \ref{def:L2defect}, the converse implication for the vanishing of the capacities also holds. More precisely, in the proof of Theorem \ref{thm_X-parab} we will establish that
$$
\CAP_f(K,U)=0
\qquad\Longrightarrow\qquad
\CAP_X(K,U)=0.
$$
\end{remark}

\section{Volume growth and $X$-parabolicity}\label{section:volume}

Let $M_f$ be a weighted manifold. From the point of view of the weighted Laplacian $\Delta_f$, the invariant measure $\vol_{f}$ is a {\it canonical measure} according to the following
\begin{lemma}\label{lem_symm_vol_f}
    Let $M_f$ be endowed with a $f$-invariant measure $\mu$ with respect to which $\Delta_f$ is symmetric, i.e., $\Delta_f^{\ast, \mu} = \Delta_f$. Then $\mu = c \cdot \vol_f$ for some constant $c >0$. Namely, $\vol_f$ is essentially the unique invariant measure for $\Delta_f$ such that $\Delta_f$ is symmetric.
\end{lemma}

\begin{proof}
    Since $\mu$ is invariant for $\Delta_f$ then $d\mu = e^{-g}\dv_f$ for some $g \in C^\infty(M)$. By assumption
    \[
    \Delta_f^{\ast,\mu} = \Delta_f
    \]
    that is, for every $\varphi \in C_c^\infty(M)$,
    \[
    \Delta_f\varphi = e^{g}\Delta_f (e^{-g}\varphi).
    \]
    This latter, in turn, can be simplified in the form
    \[
\langle \nabla g, \nabla \varphi\rangle = 0 \quad \text{for all $\varphi \in C^\infty_c(M)$}.
    \]
 Thus, $g$ is constant, and this completes the proof.
\end{proof}

The symmetry of $\Delta_f$ with respect to $\vol_f$, a trivial consequence of the (weighted) integration by parts, gives rise to the natural relations between the function, potential and spectral theory of $\Delta_f$ and the  geometry of the underlying space $M_f$. Among these important relations we have the parabolicity of complete weighted manifolds with moderate weighted volume growth.\smallskip

The main difficulty in understanding if and to what extent volumes are related to $X$-parabolicity is that we have to detect, even in this very general context, somewhat canonical measures, if any, within the class of natural measures associated with the $X$-Laplacian, namely, the class of $X$-invariant measures.\smallskip

In the following example we show that on a complete manifold a moderate  mass-growth of an $X$-invariant measure does not, in general, imply the parabolicity of the $X$-Laplacian. Therefore, the choice of the invariant measure is crucial. As we will see, this fact is already visible at the level of the $f$-Laplacian.\smallskip

\begin{example}\label{ex-1}
In this first example, we are going to consider a  warped product manifold $M = \mathbb{R}\times_\s \mathbb{S}^{m-1}$ endowed with the  metric $\metric  = dr \otimes dr + \sigma^2(r)g_{\ss^{m-1}}$ where $\s: \rr \to \rr_{>0}$ is an even function.\smallskip

Set $h(r) = \omega_m \sigma^{m-1}(r)$, where $\omega_m$ is the measure of the $(m-1)$-dimensional unit sphere in $\mathbb{R}^m$ and note that
\[
\vol(M) < \infty \Longleftrightarrow \int^\infty h(r)dr < \infty.\smallskip
\]
Next, for a smooth function $f:M\to \rr$ to be specified later, we consider the weighted manifold $M_f = (M,\metric ,\vol_f)$. Note that, if $f$ is harmonic, namely
\[
\Delta f = 0 \ \ \text{on }M,
\]
then the Riemannian measure $\vol$ of $M$ is invariant for $\Delta_f$. Indeed, in this situation,
\[
\Delta_f^{\ast,\vol} = \Delta_{-f}
\]
and the constant function $u \equiv 1$ solves $\Delta_{-f}u=0$.\smallskip

Finally, we define the weighted volume of the shells $\Omega_R = (0,R) \times \ss^{m-1}$ and $\Omega_{-R} = (-R,0) \times \ss^{m-1}$ in $M_f$ as
\[
W_f(\pm R) = \pm \int_0^{\pm R} h(r)e^{-f(r)}dr,
\]
and we introduce the following constants
\[
I_f^{\pm} = \pm \int_0^{\pm \infty} \frac{dR}{W_f'(R)}.
\]
According to Proposition \ref{prof-f-parabolic-double}, $M_f$ is parabolic if and only if $I_f^{\pm}=+\infty$.\smallskip

With this preparation, we are going to choose the warping function $\s$ and the harmonic function $f$ in such a way $\int^{+\infty} h(r) dr < +\infty$ and $I^{\pm}_f < +\infty$. This will give us that, for the $\Delta_f$-invariant measure $\mu=\vol$, we have  $\mu(M)<+\infty$ but $M$ is non-parabolic with respect to $\Delta_f$.

To this end, let $m=2$, and choose the even function $\sigma$ such that $\sigma(r) = e^{-r}$ for $r>1$. Moreover, setting $f(r) = - \int_{0}^r \sigma(t)^{-1} dt$, it is easy to see that $\Delta f = 0$. A straightforward computation shows that $I^{+}_{f} < +\infty$ so that $M_{f}$ is not $f$-parabolic. Analyzing the opposite end, we can show that $I^{-}_{- f} < +\infty$ so that $M_{f}$ is not $-f$-parabolic.  
\end{example}


In light of the above examples, we need to detect some condition on an $X$-invariant measure that gives it a somehow ``canonical" behavior with respect to parabolicity. To this end, recall that an $X$-invariant measure $\mu$ has always a smooth, positive density, namely, $d\mu = u \dv$ with $0< u \in C^{\infty}(M)$. We let $f = -\log u$ so that $d\mu = e^{-f}\dv$ and we propose the following
\begin{definition}\label{def:L2defect}
Let $d\mu=e^{-f} \dv$ be a smooth, positive invariant measure for $\Delta_X$. The $L^2$-defect for $\Delta_X$ to be symmetric with respect to $\mu$ is defined by 
    \[
    \d_f(\Delta_X) = \| X - \nabla f \|_{L^2(M,\dv_f)} = \left( \int_M |X - \nabla f |^2 \dv_f \right)^{\frac{1}{2}}.
    \]
    Say that $\mu$ is a {\it canonical measure} for $\Delta_X$ if $\d_f(\Delta_X)<+\infty$.
\end{definition}

\begin{example}[Trivial example]
    Let $\mu=e^{-f}\dv$ be an $X$-invariant measure. Then $\Delta_X$ is $\mu$-symmetric, i.e. $\Delta_X^{\ast,\mu} = \Delta_X$, if and only if $\d_f(\Delta_X)=0$ and, hence, $\Delta_X = \Delta_f$.
\end{example}

\begin{example}[Non-gradient example]
    We consider the Riemannian manifold $(\rr^m,\< \cdot,\cdot\>)$ with $m \geq 3$, $\< \cdot,\cdot\>_\b = \b^2 \< \cdot,\cdot\>_{\rr^m}$ and $0<\b \in C^\infty (\rr^m)$. We note that:
    \begin{itemize}
        \item The smooth vector field $X_x=(\a(x),0,\cdots,0)$ on $\rr^m$ is non-gradient as soon as $\b^2\a \in C^\infty(\rr^m)$ depends on some of the variables $x_2,\cdots,x_m$.
        \item The volume measure $d\mu = e^{-0} \dv$ is $X$-invariant provided
        \[
        \div X = \b^{-m}\partial_{x_1} ( \b^{m} \a) = 0.
        \]
        \item The $L^2$-defect for $\Delta_X$ to be symmetric with respect to $\vol$ (corresponding to $f \equiv 0$) satisfies $\d_0(\Delta_X)<+\infty$, and hence $\mu$ is canonical, if and only if
        \[
        \int_{\rr^m} \a^2(x) \b^{m+2}(x) dx < +\infty.
        \]
    \end{itemize}
With this preparation, if we choose $\b(x) = (1+|x|^2)^{100}$ and $\a(x) = \b(x)^{-m}$ we obtain that $(\rr^m,\metric_\b )$ is a complete manifold  whose volume measure $\vol$ is a canonical measure for $\Delta_X$.
\end{example}

\begin{example}[Non-canonical measure]
Let $(M,\metric )$ be any noncompact parabolic manifold, e.g., a complete Riemannian manifold satisfying $\vol(B_R(o))= O(R^2)$. Pick any non-constant harmonic function $g:M \to \rr$. Then, as we have already noticed, $d\mu = e^0 \dv$ is an invariant measure for $\Delta_g$. However, $\mu$ cannot be canonical, i.e., its $L^2$-defect satisfies $\d_0(\Delta_g) = +\infty$. To see this, simply note that, with $X=\nabla g$ and $f \equiv 0$, the defect reads
\[
\d_0(\Delta_g)= \left(\int_M |\nabla g|^2 \dv\right)^\frac{1}{2}
\]
so we can use the following classical result that, for possible future applications, we state in weighted form.
\begin{proposition}\label{prop:L2gradient}
Let the weighted manifold $M_f$ be $f$-parabolic. Then, for every nonconstant $f$-harmonic function $u \in C^\infty(M)$ it holds $|\nabla u| \not \in L^2(\dv_f)$.
\end{proposition}
We recall the argument for the sake of completeness. It is based on the following parabolicity criterion, see \cite[Theorem 2.19]{PS-Ensaios}. We note that the proof remains valid when $\div X$ is a Radon measure.
\begin{lemma}[Kelvin-Nevanlinna-Royden]
The weighted manifold $M_f$ is $f$-parabolic if and only if the following holds. Let $X \in L^2(\dv_f)$ be a vector field whose $f$-divergence $\div_f X$ is a signed Radon measure whose negative variation $(\div_f X)_{-}$ has finite total mass. Then 
    \[
    \int_M \div_f X \dv_f =0
    \]
    in the sense that the total mass of $ \div_f X$ is zero.
\end{lemma}
\begin{proof}[Proof (of Proposition \ref{prop:L2gradient})]
Let $\Delta_f u =0$ with $|\nabla u| \in L^2(\dv_f)$. Then, $u_+ = \max(u,0)$ satisfies $\Delta_f u_+ \geq 0$ and $|\nabla u_+|\in L^2(\dv_f)$. By the Kelvin-Nevanlinna-Royden criterion, $\Delta_f u_+ =0$. Since $u_+ \geq 0$, by $f$-parabolicity $u_+ \equiv const$. The same argument applied to $u_-=-  \min(u,0)$ gives $u_- \equiv const$. It follows that $u = u_+ - u_- \equiv const$, as desired.
        \end{proof}
\end{example}

\begin{example}[multiplicity of canonical measures]
It is easy to show that, even at the level of the weighted Laplacian $\Delta_f$, there could exist many canonical measures.

Take, for instance, a manifold $(M,\metric)$ with at least two non-parabolic ends $\E_1$ and $\E_2$. Then, it is standard to construct a non-constant harmonic function $f:M \to [0,1]$ satisfying $|\nabla f|\in L^2(M,\dv)$. Consider the corresponding $f$-Laplacian. As we have already observed, the Riemannian measure $\vol$ is $\Delta_f$-invariant and $\d_0(\Delta_f) <+\infty$. Thus, $\vol$ is canonical for $\Delta_f$. Since obviously $\vol_f$ is canonical, we have obtained a manifold supporting two canonical measures for the $f$-Laplacian. 
    
In general, by \cite[Proposition 1]{Gr1}, the existence of a positive, bounded, harmonic function with finite Dirichlet energy is equivalent to the existence of a smooth hypersurface $\Gamma$ dividing $M$ into two $D$-massive sets. Thus, these manifolds can be taken as examples of manifolds to produce non-uniqueness of canonical measures.
\end{example}

The next crucial result justifies the terminology we have introduced. 

\begin{theorem}\label{thm_X-parab}
Let $(M,\metric)$ be a complete Riemannian manifold endowed with a canonical measure $\mu= e^{-f}\vol$ for the $X$-Laplacian. If, for some origin $o \in M$,
    \[
    \int^{+\infty} \frac{dR}{\H^{m-1}_f (\partial B_R(o))} = +\infty,
    \]
then $M$ is both $X$-parabolic and $X_\mu$-parabolic where, we recall from Example \ref{example:X-inv}, $\Delta_{X_\mu} = \Delta_X^{\ast,\mu}$.
\end{theorem}
\begin{proof}
 Let $K$ be a precompact set. Fix a smooth, precompact exhaustion $U_n \uparrow M$ with $K \Subset U_1$. Since the smooth metric measure space $M_f$ is $f$-parabolic, arguing as in \cite[Lemma E.1]{BMP}, we regularize the equilibrium potentials of the $f$-capacity for the condensers $(\bar U_1,U_n)$ using the mollification described in Section \ref{section:kernels}, and deduce that there exists a sequence $\vp_j \in C^\infty_c(M)$ such that $0 \leq \vp_j \leq 1$, $\vp_j = 1$ in $\bar U_1$ and
 \[
 \int_M |\Delta_f \vp_j| \dv_f \to 0.
 \]
 Thus,
 \begin{align*}
     \CAP_{X,\mu}(K) &\leq \int_M|\Delta_X \vp_j|d\mu \\
     &= \int_M|\Delta_f \vp_j -\<X-\nabla f,\nabla \vp_j\>|\dv_f \\
     &\leq \| X - \nabla f\|_{L^2(\dv_f)} \|\nabla \vp_j\|_{L^2(\dv_f)} + \int_M |\Delta_f \vp_j|\dv_f\\
     &\leq\| X - \nabla f\|_{L^2(\dv_f)} \left(\int_M |\vp_j\Delta_f \vp_j| \dv_f \right)^{1/2} + \int_M |\Delta_f \vp_j|\dv_f\\
     &\leq \| X - \nabla f\|_{L^2(\dv_f)} \left(\int_M |\Delta_f \vp_j| \dv_f \right)^{1/2} + \int_M |\Delta_f \vp_j|\dv_f.
 \end{align*}
 By taking $j \to +\infty$, we conclude that $\CAP_{X,\mu}(K)=0$. Therefore, by Theorem \ref{thm_equiv_Xparab}, $M$ is $X$-parabolic. The $X_\mu$-parabolicity follows from Proposition \ref{equiv_green_Xparab}. Alternatively, avoiding the use of Green function characterizations of parabolicity, we can simply note that
\[
\Delta_{X_\mu} \vp = \Delta_f \vp + \<X-\nabla f, \nabla \vp \> \quad \text{with}\quad \div_f(X_\mu -\nabla f) =0.
\]
Therefore, $\mu$ is also $X_\mu$-invariant and by the exact same reasoning, it follows that $M$ is $X_\mu$-parabolic.
\end{proof}

\begin{remark}
As a scholium of the above proof, one can see that $f$-parabolicity is equivalent to the existence of a sequence of superharmonic functions $\varphi_j \in C^0_0(\Omega_j)\cap W^{1,2}(\Omega_j)$, with $\varphi_j = 1$ on $K\subset \Omega_1$, $0\leq \varphi_j  \leq 1$, and $\Delta_f \varphi_j = 0$ in the set $\Omega_j\backslash \overline{K} = \{x \colon 0<\varphi_j(x)<1\}$, such that
$$\int_{\Omega_j} \vert \nabla \varphi_j\vert^2 \dv_f \rightarrow 0 \quad \text{and} \quad \int_{\Omega_j} \vert \Delta_f \varphi_j\vert \dv_f \rightarrow 0 \quad \text{as} \ \ j \rightarrow + \infty.$$ 
\end{remark}

\begin{remark}[Growth of harmonic functions]
Let $(M,\metric )$ be any complete Riemannian manifold and pick any harmonic function  $f\not=const$ on $M$. Then $\vol_f$ is a canonical measure for $\Delta_f$. On the other hand, as we have already observed, $\vol$ is another, independent, invariant measure for this operator because $\Delta_f^{\ast,\dv}= \Delta_{-f}$. Therefore, by Proposition \ref{prop:uniqueness-Xinvariant}, the manifold is not $\Delta_{-f}$-parabolic. It follows from Theorem \ref{thm_X-parab} that
\[
\limsup_{R \to +\infty} \frac{\int_{B_R}e^{f}\dv}{R^2} = +\infty,
\]
and this is true regardless of the geometry of $(M,\metric)$ and of the choice of the non-constant harmonic function $f$.    
\end{remark}

\section{A Yau-Type $L^p$-Liouville Theorem}\label{section:Yau}

A legendary Liouville-type result by S.T. Yau in the '70s, \cite[Theorem 3]{Ya}, states that, on any complete Riemannian manifold, nonnegative subharmonic functions are necessarily constant if they are in some $L^p$-class, for $1<p<+\infty$. This follows by plugging standard cut-offs (whose existence requires completeness) into the $L^p$-Caccioppoli inequality.  Since the $f$-Laplacian is symmetric with respect to the weighted measure, the same result, with the same proof, holds for $f$-subharmonic functions on complete weighted manifolds. We are going to show how to extend this kind of conclusion to the $X$-Laplacian.\smallskip

In Definition \ref{def:L2defect} we introduced the notion of $L^2$-defect for the $X$-Laplacian to be symmetric with respect to a given invariant measure. This definition could be extended to the whole $L^p$-scale in an obvious way. Although we are merely interested in the case $p=\infty$ we put the following

\begin{definition}\label{def:Lpdefect}
    Let $d\mu=e^{-f} \dv$ be a positive invariant measure for $\Delta_X$. Given $0 <  p \leq \infty$, the $L^p$-defect for $\Delta_X$ to be symmetric with respect to $\mu$ is defined by 
    \[
    \d^p_f(\Delta_X) = \| X - \nabla f \|_{L^p(M,\dv_f)}.
    \]
\end{definition}
\begin{remark}
    As it happens for the $L^2$-defect, $\Delta_X$ is symmetric with respect to $\mu$ if and only if $\delta^p_f(\Delta_X)=0$ for some (hence any) $p \in (0,\infty]$.
\end{remark}
The Liouville theorem by Yau extends to $X$-invariant measures with finite $\infty$-defect.
\begin{theorem}\label{Yau_thm_sharp}
$\!$Let $(M,\!\metric)$ be a complete Riemannian manifold endowed with an $X\!$-invariant measure $\mu = e^{-f}\dv$, and assume that  $\delta^\infty_f (\Delta_X) < +\infty$. Then, for every $u \in C^{2}(M)$ and $1<p<\infty$,
\[
\begin{cases}
\Delta_X u \geq 0 \\
u \geq 0 \\
u \in L^{p}(M,\mu)
\end{cases}
\Longrightarrow
u \equiv const.
\]
\end{theorem}
\begin{remark}
    Actually, as it will be clear from the proof, the conclusion holds provided $\sup_{B_{2R}\backslash B_R} | X - \nabla f| = O(R)$, as $R \to +\infty$.
\end{remark}
\begin{proof}
We follow the classical proof of Yau's theorem. Let $\phi \in C^\infty_c(M)$ be a test function satisfying $\phi = 1$ in $B_r$, $\supp \phi \subset B_{2r}$ and $\Vert\nabla \phi\Vert_{L^\infty} \leq C/r$. To simplify the exposition, assume also that $u>0$ on $\supp \phi$ (otherwise, one can replace $u$ with $(u^{2}+\d^{2})^{1/2}$ or $u+\d$ and let $\d \to 0^{+})$). We then compute
\begin{eqnarray*}
0 &\leq & \int_{B_{2r}} \phi^2 u^{p-1} \Delta_X u \dv_f \\[0.2cm]
&\leq & \int_{B_{2r}} \phi^2 u^{p-1} \Delta_f u \dv_f - \int_{B_{2r}} \phi^2 u^{p-1}\< X-\nabla f,\nabla u\> \dv_f \\[0.2cm]
&\leq & -(p-1)\int_{B_{2r}} \phi^2 u^{p-2}\vert \nabla u\vert^2 \dv_f - 2\int_{B_{2r}} \phi u^{p-1} \<\nabla \phi, \nabla u\> \dv_f \\[0.2cm]
& & - \frac{1}{p} \int_{B_{2r}} \<\phi^2e^{-f}( X-\nabla f),\nabla  u^{p}\> \dv.
\end{eqnarray*}
Applying Young inequality and the fact that $\mu$ is $X$-invariant so that $\div(e^{-f}(X-\nabla f)) = 0$, and rearranging we obtain
\begin{eqnarray*}
\frac{p-1}{2}\int_{B_{r}}  u^{p-2}\vert \nabla u\vert^2 \dv_f &\leq & \frac{2}{p-1} \int_{B_{2r}} u^p \vert \nabla \phi\vert^2 \dv_f + \frac{2}{p}\int_{B_{2r}} \phi u^{p}\< \nabla \phi, X- \nabla f\> \dv_f\\[0.2cm]
&\leq & \left(\frac{2C^2}{(p-1)r^2} + \frac{2C}{pr}\sup_{B_{2r}\backslash B_r}\vert X-\nabla f\vert \right)\int_{B_{2r}\backslash B_r} u^{p} \dv_f.
\end{eqnarray*}
The result follows by letting $r\rightarrow +\infty$.
\end{proof}

\begin{example}\label{sharpness_yau}
We are going to show that an $L^\infty$-defect condition cannot be avoided. To this end we consider the example introduced in \cite[Proposition 4.8]{BRS}. Let $(\mathbb{R}^m, \metric )$ be the standard Euclidean space, and let $X$ be the smooth vector field given by
$$X^i =  x_i + 2x_{\sigma(i)} e^{\frac{x_i^2 - x_{\sigma(i)}^2}{2}},$$
where $\sigma \colon \{1,\ldots,m\} \rightarrow \{1,\ldots,m\}$ is a permutation satisfying $\sigma(i) \not= i$ for all $i\in \{1,\ldots,m\}$.

The Gaussian measure $d\mu_1 = c_m e^{-\frac{\vert x\vert^2}{2}}\dv$ and the measure $d\mu_2 = u \dv$ with 
$$u(x) = d_m e^{-\frac{\vert x\vert^2}{2}}\sum_{i=1}^{m}\int_{-\infty}^{x_i} e^{-\frac{s^2}{2}}ds,$$
are both $X$-invariant probability measures.
We note that, by writing $u = e^{-f}$,
\begin{eqnarray*}
\vert X^i - \partial_i f\vert &=& \left\vert  x_i + 2x_{\sigma(i)} e^{\frac{x_i^2 - x_{\sigma(i)}^2}{2}}  - x_i + \frac{e^{-\frac{x_i^2}{2}}}{\sum_{i=1}^m \int_{-\infty}^{x_i} e^{-\frac{s^2}{2}}ds}\right\vert \\[0.2cm]
&\geq &  2x_{\sigma(i)} e^{\frac{x_i^2 - x_{\sigma(i)}^2}{2}} - \frac{e^{-\frac{x_i^2}{2}}}{\sum_{i=1}^m \int_{-\infty}^{x_i} e^{-\frac{s^2}{2}}ds}.
\end{eqnarray*}
It then follows that the growth condition on the $L^\infty$-defect is violated. In view of Doob's uniqueness result, see Theorem \ref{th-Doob} and Corollary \ref{cor-unique-XLap}, the heat semigroup $P_t$ generated by $\Delta_X$ has to be non-conservative. Therefore, combining Theorem \ref{prop-abstractexistence} with Theorem \ref{th-Doob}, we conclude that the operator $\Delta_X$ does not satisfy the $\lambda$-Liouville property, and thus, there exists a (nontrivial) bounded solution $0<v \in C^\infty(\rr^m)$ to the equation
\[
\Delta_X v = \lambda v \geq 0.
\]
Since $\mu_2(\rr^m) < +\infty$ and $v$ is bounded, for every $p \in  (1,+\infty)$ we have 
\[
\int_{\rr^m} v^p d\mu_2 < \infty.
\]
\end{example}

\section{Geometric applications}\label{section:geometry}

Let $(M,\metric )$ be a complete, $m$-dimensional Riemannian manifold endowed with a smooth vector field $X$. Having fixed an origin $o \in M$ we denote by $\rho(x)$ the intrinsic distance function from $o$. Finally, we introduce the $X$-Ricci tensor
\[
\ric_X = \ric + \frac{1}{2}L_X \metric 
\]
and, for $n>m$, its modified version
\[
\ric_{X,n} =\ric_X - \frac{1}{n-m} X^{\flat} \otimes X^{\flat}.
\]
Here, the symbol $L$ stands for the Lie-derivative of $2$-tensor fields and $\flat$ is the standard musical isomorphism. In the next result we collect known Laplacian comparison results for $\rho$ under radial lower bounds of either of these tensors.
\begin{theorem}\label{th:laplaciancomparison}
The following Laplacian comparison estimates hold true.
    \begin{itemize}
    \item [(a)] Assume that $\ric_{X,n} \geq (n-1) K(\rho)$, for some $n>m$. If $h \in C^1([0,+\infty))$ is a positive solution of
    \begin{equation}\label{solution-h}
    \begin{cases}
        h'' + Kh = 0 \ \ \text{in} \ (0,+\infty); \\
        \lim_{r\rightarrow 0} r \frac{h'}{h}(r) = 1,
    \end{cases}    
    \end{equation}
then
    \begin{equation}\label{ricXn-laplaciancomparison}
    \Delta_X \rho \leq (n-1) \frac{h'}{h}(\rho)    
    \end{equation}
    pointwise on $\ M\backslash (\cut (o) \cup \{o\})$ (see \cite{BQ}).
    
    \item [(b)] Assume that $\ric_X \geq - (m-1)F(\rho)$, for some $F \in C^0([0,+\infty))$. Then, for every $0<\delta < \min(\inj(o),1)$, it holds
    \begin{equation}\label{ricX-laplaciancomparison}
        \Delta_X\rho \leq C + (m-1)\int_\delta^\rho F(t)dt
    \end{equation}
    pointwise on $(M\backslash B_\delta(o))\backslash \cut(o)$, with $C = \max_{\partial B_{\delta}(o)} \Delta_X \rho$ (cf. \cite{CMMR}).
\end{itemize}
\end{theorem}

In the following, we strengthen previous results in \cite{BRW}, see Example 3.3 therein, where an uncoupled condition on the Ricci curvature and the vector field was used. Moreover, we complete their statement by adding uniqueness and potential theoretical conclusions.

\begin{corollary}\label{cor_unique_meas_Ric_positive}
Let $(M,\metric )$ be a complete Riemannian manifold endowed with a smooth vector field $X$. If $\text{Ric}_X \geq (m-1)K$, for some constant $K>0$, then there exists an essentially unique invariant measure $\mu$ for the $X$-Laplacian. Moreover, $\mu(M)<+\infty$ and $M$ is both $X$-parabolic and $X_\mu$-parabolic.
\end{corollary}
\begin{proof}
Thanks to Theorem \ref{th:laplaciancomparison} item $(b)$ with $F \equiv -K$ we can apply \cite{BRW} to deduce that $M$ supports an invariant probability measure $\mu$. On the other hand, again from Theorem \ref{th:laplaciancomparison} $(b)$, the distance function meets the Khasminskii test for $X$-parabolicity. It follows from Proposition \ref{equiv_green_Xparab} that $M$ is also $X_\mu$-parabolic which, by Lemma \ref{th-unique-liouville}, gives the uniqueness of $\mu$.  
\end{proof}

\begin{example}[Shrinking Ricci Solitons]
Let $(M,\metric, X)$ be a complete shrinking Ricci soliton with soliton constant $\lambda > 0$, that is, $\ric + \frac{1}{2}L_X\metric = \lambda\metric$. Then there exists an essentially unique invariant measure for $\Delta_X$, which is a probability measure.
\end{example}

It is by now well understood how Laplacian comparison results give rise to corresponding volume comparison theorems. In the $X$-setting, volumes have to be computed with respect to invariant measures. What is striking here is that the volume comparisons hold for any choice of the $X$-invariant measure see \cite{BQ} for a discussion of this phenomenon. We point out the following

\begin{theorem}\label{th:volumecomparison}
Let $(M,\metric )$ be a complete, $m$-dimensional Riemannian manifold endowed with a smooth vector field $X$. Set $\rho(x)$ for the intrinsic distance function from a reference origin $o \in M$. 
Assume that one of the following conditions is satisfied:
\begin{itemize}
    \item [(a)] $ \ric_{X,n} \geq (n-1) K$ for some constant $K \in \rr$ and some virtual dimension $n>m$.\medskip
    \item [(b)] $ \ric_X \geq (m-1)K$ for some constant $K > 0$.
\end{itemize}
Then, correspondingly, the following volume estimates hold for every $X$-invariant measure $d\mu=e^{-f}\dv$:
\begin{itemize}
    \item [(A)] The function
    \[
    R \mapsto \frac{\H^{n-1}_f(\partial B_R(o))}{h(R)^{n-1}}
    \]
    is non-increasing for a.e. $R>1$, where $h \in C^1([0,+\infty))$ solves \eqref{solution-h}.\medskip
    
    \item [(B)] $\mu(M) < +\infty$.
\end{itemize}
\end{theorem}

\begin{proof}
    (A) To begin with we observe that, by the very result of Calabi, \cite{Ca}, \eqref{ricXn-laplaciancomparison} with $K(\rho)\equiv K$ extends to the whole manifold in the barrier sense, hence in the viscosity sense. This implies, by a celebrated result of H. Ishii, \cite{Is}, that \eqref{ricXn-laplaciancomparison} holds in the sense of distributions with respect to the given invariant measure; see Appendix \ref{appendix-Ishi}. In summary, for every $0\leq \vp \in C^\infty_c(M)$,
    \begin{equation}\label{laplaciancomp-distr}
       \int_M \rho \Delta_X^{\ast,\mu}\varphi d\mu \leq (n-1)\int_M \varphi \frac{h'}{h}(\rho)d\mu. 
    \end{equation}
    Now recall that, since $d\mu = \dv_f = e^{-f}\dv$ is an $X$-invariant measure, then
    \[
    \div\bigl(e^{-f}(X-\nabla f)\bigr)=0
    \]
    and, therefore,
    \[
    \Delta_X^{\ast,\mu}\varphi = \Delta_f \varphi + \langle X - \nabla f, \nabla \varphi\rangle.
    \]
    Thus, we can write the left-hand side of inequality \eqref{laplaciancomp-distr} as
\begin{eqnarray*}
    \int_M \rho \Delta_X^{\ast,\mu}\varphi d\mu &=& \int_M \rho \Delta_f \varphi \dv_f  + \int_M \rho \langle X - \nabla f, \nabla \varphi \rangle \dv_f \\[0.2cm]
    &=& -\int_M \< \nabla \rho, \nabla \varphi\> \dv_f +  \int_M \rho \langle X - \nabla f, \nabla \varphi \rangle \dv_f .
\end{eqnarray*}
It follows that, for every $0\leq \varphi \in C^\infty_c(M)$
$$ -\int_M \< \nabla \rho, \nabla \varphi\> \dv_f +  \int_M \rho \langle X - \nabla f, \nabla \varphi \rangle \dv_f \leq (n-1)\int_M \varphi \frac{h'}{h}(\rho)\dv_f .$$
Clearly, by density, this inequality extends to every $0 \leq \vp \in \Lip_{loc}(M)$.

Now, using the co-area formula and the divergence theorem in a Lipschitz setting, see Appendix \ref{appendix-divergenceth}, we get, for a radial $\vp=\vp(\rho)$,
\begin{eqnarray*}
\int_M \rho \langle X - \nabla f, \nabla \varphi \rangle \dv_f &=& \int_0^\infty \int_{\partial B_\rho(o)} \varphi'(\rho)\rho \langle X - \nabla f, \nabla \rho \rangle e^{-f} d \H^{n-1} d\rho \\[0.2cm]
&=& \int_0^\infty \varphi'(\rho)\rho \int_{\partial B_\rho(o)} \langle X - \nabla f, \nabla \rho \rangle d \H^{n-1}_f d\rho \\[0.2cm]
&=& \int_0^\infty \varphi'(\rho)\rho \int_{B_\rho} \div_f( X - \nabla f) \dv_f d\rho  \\[0.2cm]
&=& 0.
\end{eqnarray*}
We have thus deduced that, for all radial cut-offs $0\leq \vp(\rho) \in \Lip_{loc}(M)$,
$$ -\int_M \< \nabla \rho, \nabla \varphi\> \dv_f  \leq (n-1)\int_M \varphi \frac{h'}{h}(\rho)\dv_f .$$
Once this inequality is obtained, the validity of the Bishop-Gromov property follows from standard arguments, see e.g. \cite[Theorem 2.14]{PRS}.\medskip

\noindent (B) Inequality \eqref{ricX-laplaciancomparison} can be written in a less precise but more stable form. First of all, let
    \[
    0 < \delta := \min\left( \frac{1}{100}\inj(\bar B_1(o)),\frac{1}{100} \right) \quad \text{and}\quad \ric\geq-(m-1)S^2 \text{ on }\bar B_1(o).
    \]
For every $z \in \bar B_{2\delta}(o)$ we clearly have $0<\delta < \inj(z)$. Moreover, having set $\rho_z=d(z,\cdot)$, by the ordinary Laplacian comparison theorem we have
\[
\Delta_X \rho_z \leq D:= (m-1) \operatorname{ctn}_{-S}(\delta)+\max_{\bar B_1(o)}|X| \quad \text{on }\partial B_\delta(z).
\]
With this preparation, it is easy to show that the inequality
\begin{equation}\label{ricX-laplaciancomparison2}
\Delta_X \rho \leq D-(m-1)K(\rho - \delta),    
\end{equation}
which is clearly valid in the pointwise sense on $(M\backslash \bar B_\delta(o)) \backslash \cut(o)$, can be extended  in the barrier sense inside the cut locus. To see this, let $x \in \cut(o)\cap (M\backslash \bar B_\delta(o))$. Fix any $0 < \varepsilon \ll 1$, choose a minimizing geodesic $\gamma_x$ from $o$ to $x$ parametrized by arclength, and set $z_\varepsilon = \gamma_x(\varepsilon) \in \partial B_\varepsilon(o)$. Since $x \not \in \cut (z_\varepsilon)$, then in a neighborhood of $x$, we have,
\begin{align*}
\Delta_X \rho_{z_\varepsilon} &\leq \max_{\partial B_\delta(z_\varepsilon)} \Delta_X \rho_{z_\varepsilon} -(m-1)K(\rho_\varepsilon-\delta)\\
&\leq D-(m-1)K(\rho-\delta)+(m-1)K\varepsilon,
\end{align*}
as required.

At this point, we know that \eqref{ricX-laplaciancomparison2} holds in the distributional sense with respect to the invariant measure $d\mu=e^{-f}\dv$. Namely, for every $0 \leq \vp \in C^\infty_c(M\backslash \bar B_\delta(o))$,
\begin{equation}\label{RicX-distributionalsolution}
    \int_{M\backslash \bar B_\delta(o)} \rho \Delta_X^{\ast\mu}\vp \, d\mu \leq \int_{M\backslash \bar B_\delta(o)} \left(D-(m-1)K(\rho - \delta)\right) \vp \, d\mu.
\end{equation}
Whence, arguing as in (A), we obtain that for all radial cut-offs $0\leq \vp(\rho) \in \Lip_{loc}(M \backslash \bar B_\d(o))$
\[
-\int_{M\backslash \bar B_\delta(o)} \< \nabla \rho, \nabla \varphi\> \dv_f  \leq \int_{M\backslash \bar B_\delta(o)} \left(D-(m-1)K(\rho - \delta)\right)  \vp(\rho) \dv_f .
\]
Note that the function $G(t) = D-(m-1)K(t - \delta)$ satisfies $G(t) \leq 0 $ for all $t \geq R_0$ where
\[
R_0 = \frac{D+(m-1)K\delta}{(m-1)K}>\delta.
\]
We plug  $\vp=\vp_\varepsilon$ into the above, where
\[
\vp_\varepsilon(t)=
\begin{cases}
    0 &\text{in} \  [0,\delta], \\
    (t-\delta) & \text{in} \ [\delta,\delta+1],\\
    1 &\text{in} \  [\delta+1,R] ,\\
    \frac{R+\varepsilon-\rho}{\varepsilon} & \text{in} \ [R,R+\varepsilon],\\
    0 &[R+\varepsilon,+\infty).
\end{cases}
\]
Letting $\varepsilon \to 0^+$ and using the co-area formula we deduce that, for all $R>R_0 +1$,
\begin{align*}
\H^{m-1}_f(\partial B_R(o)) &\leq \mu(B_{\delta +1}(o) \backslash B_\delta(o)) + \int_{B_{R_0}(o) \backslash B_\delta(o)}G(\rho)d\mu + \int_{B_R(o) \backslash B_{R_0+1}(o)} G(\rho)d\mu\\
&\leq \mu(B_{\delta +1}(o) \backslash B_\delta(o)) + \int_{B_{R_0}(o) \backslash B_\delta(o)}G(\rho)d\mu + G(R_0+1) \mu(B_R(o)\backslash B_{R_0+1}(o))\\
&=\mu(B_{\delta +1}(o) \backslash B_\delta(o)) + \int_{B_{R_0}(o) \backslash B_\delta(o)}G(\rho)d\mu -G(R_0+1)\mu(B_{R_0+1}(o)) \\ & + G(R_0+1) \vol_f(B_R(o)).
\end{align*}
This latter can be written in the form
\[
\frac{d}{dR}\vol_f(B_R(o)) \leq \alpha-\beta \vol_f(B_R(o)), \, \forall R>R_0+1
\]
with $\alpha,\beta>0$. Whence, the desired conclusion follows.
\end{proof}

It is remarkable that Theorem~\ref{th:volumecomparison} provides the finiteness of every invariant measure $\mu$ without assuming that $\mu$ is canonical. As a consequence, we have the following 

\begin{corollary}
Let $(M,\metric )$ be a complete Riemannian manifold endowed with a smooth vector field $X$ satisfying $\div X = 0$. If $\text{Ric}_X \geq (m-1)K>0$ then $\vol(M)<+\infty$.
\end{corollary}

\begin{remark}
Combining Theorem \ref{thm_X-parab} with Theorem \ref{th:volumecomparison}, we obtain an analog of Corollary \ref{cor_unique_meas_Ric_positive} without appealing to the existence of a Lyapunov function, under the additional assumption that there exists a canonical $X$-invariant measure $\mu$. Indeed, in this case, the condition $\text{Ric}_X \geq (m-1)K>0$ implies that $\mu(M)<+\infty$, while Theorem \ref{thm_X-parab} yields the parabolicity of $M$ with respect to both $X$ and $X_\mu$. This naturally raises the question of whether the existence of such a canonical invariant measure can itself be deduced from suitable geometric assumptions on the weighted manifold. We shall address this question in a forthcoming paper.
\end{remark}

\noindent{\bf Acknowledgments.} The authors are grateful to Francesco Caravenna and Franco Flandoli for their interest in this work and for drawing attention to relevant references on invariant measures. The first author would like to thank the Department of Mathematics of the Università degli Studi di Milano-Bicocca for the kind hospitality during his stay in summer 2025, where part of this work was completed. L.F.P. is partially supported by CNPq, Brazil, under grant numbers 306543/2022-2, 402563/2023-9 and 402713/2024-9. S.P. is partially supported by the Italian group INdAM-GNAMPA. A.G.S. is a member of the Italian INdAM-GNAMPA group "Equazioni Differenziali e Sistemi Dinamici".
\medskip

\noindent{\bf Declaration of AI use.} No AI tools were used in the preparation of this manuscript except for a final check for misprints and  notational inconsistencies.

\appendix

\section{A criterion of $f$-parabolicity in warped products}

We are going to address the problem of characterizing the $f$-parabolicity of a warped product manifold $M_\sigma = \mathbb{R}\times_\s \mathbb{S}^{m-1}$ endowed with the  metric $\metric  = dr \otimes dr + \sigma^2(r)g_{\ss^{m-1}}$.

\begin{proposition}\label{prof-f-parabolic-double}
Let $(M_\sigma,e^{-f(r)}vol)$ be a weighted manifold and assume that $\sigma$ and $f$ are even. Then the manifold $(M_\sigma,e^{-f}vol)$ is $f$-parabolic if and only if $I_f = \infty$.
\end{proposition}
\begin{proof}
On the set $\{r>0\}$ we will consider the solution of the $f$-Laplacian equation
$$ v'' + \frac{h'(r)}{h(r)}v' - v' f' = 0.$$
It is easy to see that the solutions of the above ODE are of the form
$$ v(r) = c_1 \int_0^r \frac{e^{f(r)}}{h(r)}dr + c_2,$$
for some constants $c_1, c_2 \in \mathbb{R}$. This function can be evenly extended to the entire manifold $M_\sigma$, and mollified inside  the compact set $K \doteq [-1,1]\times \mathbb{S}^{m-1}$. Choosing $c_1>0$ one can see that $\Delta_f v = 0$ outside of $K$ and satisfies $v(x) \rightarrow \infty$ as $x \rightarrow \infty$ whenever $I_f = \infty$. Thus, by Khasminskii test, $I_f = \infty$ implies that $(M_\sigma,e^{-f(r)}vol)$ is $f$-parabolic. On the other hand, still for $c_1>0$, we have that $v$ is increasing for $r>0$ and we can consider $u = \max\{v,c\}$, for some constant $c>0$ such that the resulting function $u$ is $f$-subharmonic on $M_\sigma$. In this case, if $I_f < \infty$ then $u$ is a bounded, $f$-subharmonic function and, therefore, $(M_\sigma,e^{-f(r)}vol)$ is not $f$-parabolic.
\end{proof}



\section{The divergence theorem on metric balls}\label{appendix-divergenceth}
In this appendix, we briefly review why the divergence theorem for smooth vector fields extends to the case of metric balls in a complete Riemannian manifold $(M,\metric )$. Some further remarks in the case of continuous vector fields will be also presented.

First of all, we recall that a domain $\Omega \subseteq M$ is said to be (locally) Lipschitz if every point $x \in \partial \Omega$ has a neighborhood $U_x$ such that $U_x \cap \Omega$ is described as the subgraph of a (locally) Lipschitz function. In the smooth category, thanks to the implicit function theorem, level subsets give rise to smooth domains as soon as the boundary level set corresponds to a regular value of the function. And this, by Sard, happens for a.e. value in the co-domain. Something similar holds in the Lipschitz category: thanks to Clarke's implicit function theorem, \cite{Cl}, together with the Lipschitz version of Sard theorem for the distance function (from a submanifold) we have the following fundamental result by L. Rifford, \cite{Ri}.
\begin{theorem}
Let $(M,\metric )$ be a complete Riemannian manifold. For a fixed reference origin $o \in M$, let $\rho(x)=d(x,o)$ denote the intrinsic distance function from $o$. Then, for a.e. $R>0$ the closed metric ball $\bar B_R(o)$ is a Lipschitz domain.
\end{theorem}

At this point, to conclude the validity of the divergence theorem, we can appeal to classical results in Geometric Measure Theory; see e.g. \cite[Theorem 4.6]{EG}. We state it in the weighted version that is employed through the paper.

\begin{theorem}\label{th:lipdiv}
    Let $M_f = (M,\metric ,\dv_f)$ be a complete weighted Riemannian manifold with a fixed reference origin $o \in M$. Then, given a $C^1$ vector field $X$ on $M$, for a.e. $R>0$ it holds
    \begin{equation}\label{lipdivformula}
        \int_{\bar{B}_R(o)} \div_f X\, \dv_f= \int_{\partial B_R(o)} \<X,\nabla r\> \,d\H^{m-1}_f.
    \end{equation}
\end{theorem}
Finally we recall that, with the procedure introduced in \cite{RS}, the validity of \eqref{lipdivformula} can be also extended to vector fields $X\in L^1_{loc}(M)$ thanks to the weak definition of divergence and the co-area formula. In fact, in this case, the left hand side of \eqref{lipdivformula} is replaced by
\[
\lim_{\varepsilon \to 0^+}(\div_f X,  \psi_\varepsilon) = \lim_{\varepsilon \to 0^+} - \int_{M} \< X, \nabla \psi_\varepsilon\> \dv_f,
\]
where $\psi_\varepsilon$ is the Lipschitz function defined by
\[
\begin{cases}
    1 & \text{in} \ B_{R -\varepsilon}(o), \medskip \\
    \dfrac{1}{\varepsilon}(R-r(x)) & \text{in} \ B_R(o) \backslash B_{R-\varepsilon}(o) ,\medskip \\
    0 & M\backslash B_R(o).
    
\end{cases}
\]

\section{Different notions of solutions}\label{appendix-Ishi}
On a Riemannian manifold $(M,\metric )$ endowed with a smooth vector field $X$, the differential inequality
\begin{equation}\label{diffineq}
\Delta_X u \leq h,    
\end{equation}
with $h \in C^0(M)$, can be interpreted in different ways according to the regularity of $u$. We are going to recall the main notions and how they fit into a certain hierarchy. 
\begin{definition}[Weak]
    The function $u \in W^{1,2}_{loc}(M)$ satisfies \eqref{diffineq} on $M$ in the weak sense with respect to the measure $\mu = \vol_f$ if, for all $0 \leq \vp \in C^{\infty}_c(M)$,
    \[
    -\int_M\<\nabla u,\nabla \vp\> \dv_f + \int_M \div_f (X)\,  u \dv_f\leq \int_M  h\vp \dv_f.
    \]
\end{definition}
\begin{definition}[Distributional]
    The function $u \in L^1_{loc}(M)$ satisfies \eqref{diffineq} on $M$ in the distributional sense with respect to the measure $\mu=\vol_f$ if, for all $0 \leq \vp \in C^{\infty}_c(M)$,
    \[
    \int_M u \Delta_X^{\ast,\mu}\vp \, d\mu \leq \int_M h \vp \, d\mu.
    \]
\end{definition}
\begin{definition}[Barrier]
    The function $u \in \text{LSC}(M)$ satisfies \eqref{diffineq} on $M$ in the barrier sense (of Calabi, \cite{Ca}) if, for every $x_0 \in M$ and for $\varepsilon>0$, there exist a neighborhood $U_{x_0,\varepsilon}$ of $x_0$ and a function $u_{x_0,\varepsilon} \in C^2(U_{x_0,\varepsilon})$ satisfying the following conditions:\medskip
    \begin{itemize}
        \item $u(x)-u_{x_0,\varepsilon}(x) \leq u(x_0)-u_{x_0,\varepsilon}(x_0)$ in $U_{x_0,\varepsilon}$;\medskip
        \item $\Delta_Xu_{x_0,\varepsilon} \leq h+\varepsilon$ at $x_0$.
    \end{itemize}
\end{definition}
\begin{definition}[Viscosity]
    The function $u \in C^0(M)$ satisfies \eqref{diffineq} in $M$ in the viscosity sense if, for every $x_0 \in M$ the following is satisfied: for every function $u_{x_0} \in C^2(U_{x_0})$ in a neighborhood $U_{x_0}$ of $x_0$ satisfying
\[
u(x)-u_{x_0}(x) \geq u(x_0)-u_{x_0}(x_0) \text{ in } U_{x_0}
\]
we have
\[
\Delta_Xu_{x_0} \leq h \text{ at } x_0.
\]
\end{definition}

It is clear from the very definitions that pointwise solutions are solutions in all the other senses. Moreover, weak solutions are distributional solutions as a consequence of the divergence theorem, and barrier solutions are viscosity solutions. In the next result, we recall why viscosity solutions are distributional solutions with respect to any given measure $\mu=\vol_f$.

\begin{theorem}\label{th:viscositydistributional}
Let $(M,\metric )$ be a Riemannian manifold, and let $u \in C^0(M)$ be a viscosity solution of the differential inequality \eqref{diffineq}.
 Then, for any measure $d\mu = e^{-f}\dv$, the following distributional inequality
\begin{equation}\label{diffineq-distr}
    \int_M u \Delta_X^{\ast,\mu}\vp \,d\mu \leq \int_M h\vp\, d\mu,
\end{equation}
holds for every $0 \leq \vp \in C^\infty_c(M)$.
\end{theorem}

\begin{proof}
We preliminarily observe that
\begin{align*}
    e^{-f}\Delta_X^{\ast,\mu} = (e^{-f}\Delta_X)^{\ast,\vol}.
\end{align*}
Therefore \eqref{diffineq-distr} writes as
\begin{align*}
    \int_M u \, (e^{-f}\Delta_X)^{\ast,\vol}\vp \dv \leq \int_M h \vp e^{-f}\dv.
\end{align*}
Now, pick a test function $0 \leq \vp \in C^\infty_c(M)$. Without loss of generality, up to using a partition of unity, we can assume that $\supp \vp$ is contained in a coordinate domain of $M$. We are thus reduced to proving that
    \begin{align*}
    \int_{\rr^m} u \, \sqrt{g}(e^{-f}\Delta_X)^{\ast,\vol}\vp \,  \, dx \leq \int_{\rr^m} \sqrt{g} h \, \vp \, e^{-f}\, dx,
\end{align*}
with $g$ the determinant of the coefficient matrix of $\metric $. Moreover, within the same coordinate domain,
\begin{align*}
    \sqrt{g}(e^{-f}\Delta_X)^{\ast,\vol} = (\sqrt{g}e^{-f} \Delta_X)^{\ast,dx}.
\end{align*}
It follows that what we have to prove is the validity of the following inequality
\begin{equation} \label{diffineq-distr-Rn}
    \int_{\rr^m} u (\sqrt{g}e^{-f} \Delta_X)^{\ast,dx}\vp\,  dx \leq \int_{\rr^m} \sqrt{g}  e^{-f} h \, \vp \, dx.
\end{equation}
By Ishii's famous equivalence \cite[Theorem 1]{Is}, \eqref{diffineq-distr-Rn} holds true provided  the following inequality
\begin{equation}\label{diffineq-visc-Rn}
    \sqrt{g}e^{-f} \Delta_X u \leq \sqrt{g}e^{-f}h
\end{equation}
holds in the viscosity sense, and this is surely the case thanks to the fact that \eqref{diffineq} holds in the viscosity sense.
\end{proof}

\end{document}